\documentclass[10pt]{article}
\usepackage{amssymb,amsmath,amsthm}
\usepackage{hyperref}
\usepackage{verbatim}
\usepackage{graphicx,color}
\newtheorem*{definition}{Definition}
\newtheorem{theorem}{Theorem}
\newtheorem{theorem1}{Theorem}
\newtheorem{conjecture2}[theorem1]{Conjecture}
\newtheorem{theorem3}[theorem1]{Theorem}
\newtheorem{lemma}[theorem]{Lemma}
\newtheorem*{fact}{Fact}
\newtheorem{claim}{Claim}
\newtheorem{conjecture}[theorem]{Conjecture}
\newtheorem{proposition}[theorem]{Proposition}
\newtheorem{corollary}[theorem]{Corollary}

\theoremstyle{definition}
\newtheorem{remark}{Remark}
\newcommand{\ud}{\,\mathrm{d}}
\DeclareMathOperator{\Cay}{Cay}

\newcommand{\perm}{\textrm{perm}}
\newcommand{\domgeq}{\trianglerighteq}

\newenvironment{myproof}[1][\proofname]{\proof[#1]}{\endproof}
\newcommand{\round}[1]{\operatorname{round}(#1)}

\newcommand{\beq}[1]{\begin{equation}\label{#1}}
\newcommand{\enq}[0]{\end{equation}}

\begin{document}
\title{A quasi-stability result for dictatorships in $S_n$ }
\author{David Ellis, Yuval Filmus\footnote{Supported by the Canadian Friends of the Hebrew University / University of Toronto Permanent Endowment.}, and Ehud Friedgut\footnote{Supported in part by I.S.F. grant 0398246, and BSF grant 2010247.}}
\date{December 2013}
\maketitle

\begin{abstract}
We prove that Boolean functions on \(S_n\) whose Fourier transform is highly concentrated on the first two irreducible representations of \(S_n\), are close to being unions of cosets of point-stabilizers. We use this to give a natural proof of a stability result on intersecting families of permutations, originally conjectured by Cameron and Ku \cite{cameron}, and first proved in \cite{cameronkuconj}. We also use it to prove a `quasi-stability' result for an edge-isoperimetric inequality in the transposition graph on \(S_n\), namely that subsets of \(S_n\) with small edge-boundary in the transposition graph are close to being unions of cosets of point-stabilizers.
\end{abstract}

\section{Introduction}
In extremal combinatorics, we are typically interested in subsets \(S\) of a finite set \(X\) which satisfy some property, \(P\) say. Often, we wish to determine the maximum or minimum possible size of a subset \(S \subset X\) which has the property \(P\). The maximum or minimum-sized subsets of \(X\) with the property \(P\) are called the {\em extremal} sets.

In the past fifty years, discrete Fourier analysis has been used to solve a number of extremal problems where the set \(X\) may be given the structure of a finite group, \(G\). In this case, given a set \(S\) with the property \(P\), one may consider the {\em characteristic function} \({\bf{1}}_{S}\), and take the Fourier transform of \({\bf{1}}_S\).  (Recall that the characteristic function of a subset \(S \subset G\) is the real-valued function on \(G\) with value \(1\) on \(S\) and \(0\) elsewhere.) If we are lucky, the property \(P\) gives us information about the Fourier transform of \({\bf{1}}_{S}\), which can then be used to obtain a sharp bound on \(|S|\). Often, such proofs tell us that if \(S\) is an extremal set, then the Fourier transform of \({\bf{1}}_{S}\) must be supported on a certain set, \(T\) say; this can then be used to describe the structure of the extremal sets.

Under these conditions, it often turns out that if \(S\) is `{\em almost-extremal}', meaning that it has size close to the extremal size, then the Fourier transform of \({\bf{1}}_S\) is highly concentrated on \(T\). If one can characterize the Boolean functions whose Fourier transform is highly concentrated on \(T\), one can describe the structure of the almost-extremal sets. Sometimes, almost-extremal sets must be close in structure to a genuine extremal set; this phenomenon is known as {\em stability}. 

Characterizing the Boolean functions whose Fourier transform is highly concentrated on \(T\) often turns out to be a hard problem. To date, such a characterization has been obtained in several cases where the group \(G\) is Abelian, using the well-developed theory of Fourier analysis on Abelian groups. The simplest case is that of  dictatorships: Friedgut, Kalai and Naor \cite{fkn} prove that a Boolean function on \(\{0,1\}^n\) whose Fourier transform is close to being concentrated on the first two levels, must be close to a dictatorship (a function determined by just one coordinate). This was useful for Kalai in \cite{KalaiArrow} where he deduced a stability version of Arrow's theorem on social choice functions, namely that if a neutral social choice function has small probability of irrationality, then it must be close to a dictatorship.
 Alon, Dinur, Friedgut and Sudakov \cite{adfs} proved a similar result for $\mathbb{Z}_r^n$ (a result later improved by  Hatami and Ghandehari, \cite{hatami}), and utilized it to describe the large independent sets in powers of a large family of graphs.

In this paper, we obtain a similar result for Boolean functions on \(S_n\). It is easy to see that if $f\colon S_n \to \mathbb{R}$, then the Fourier transform of $f$ is supported on the first two irreducible representations of $S_n$ if and only if it lies in the subspace spanned by the characteristic functions of cosets of point-stabilisers. (For brevity, we refer to the cosets of point-stabilisers as {\em 1-cosets}, and we denote the subspace spanned by their characteristic functions as $U_1$.  Similarly, a {\em $t$-coset} is a coset of the stabilizer of an ordered $t$-tuple of distinct points.) If $f$ is Boolean, i.e. $f\colon S_n \to \{0,1\}$, and $f \in U_1$, then $f$ is the characteristic function of a disjoint union of 1-cosets. (This is somewhat trickier to show; a proof may be found e.g.\ in \cite{tintersecting}.) A disjoint union of 1-cosets is precisely a subset of $S_n$ whose characteristic function is determined by the image of the pre-image of just one coordinate; by analogy with the $\{0,1\}^n$ case, we call these subsets (or their characteristic functions) {\em dictators}.

In this paper, we consider Boolean functions on $S_n$ whose Fourier transform is {\em highly concentrated} on the first two irreducible representations of $S_n$ --- equivalently, Boolean functions which are {\em close} (in Euclidean distance) to the subspace $U_1$. We prove the following `quasi-stability' result.

\begin{theorem}
\label{thm:main}
There exist absolute constants \(C_0,\epsilon_0 >0\) such that the following holds. Let \(\mathcal{A} \subset S_n\) with \(|\mathcal{A}| = c(n-1)!\), where \(c \leq n/2\), and let \(f = {\bf{1}}_{\mathcal{A}} \colon S_n \to \{0,1\}\) be the characteristic function of \(\mathcal{A}\), so that \(\mathbb{E}[f] = c/n\). Let \(f_1\) denote the orthogonal projection of \(f\) onto \(U_1\). If \(\mathbb{E}[(f-f_1)^2] \leq \epsilon c/n\), where \(\epsilon \leq \epsilon_0\), then there exists a Boolean function \(\tilde{f}:S_n \to \{0,1\}\) such that
\begin{equation}\label{eq:expbound} \mathbb{E}[(f-\tilde{f})^2] \leq C_0c^2 (\epsilon^{1/2} + 1/n)/n,\end{equation}
and \(\tilde{f}\) is the characteristic function of a union of \(\round{c}\) 1-cosets of \(S_n\). Moreover, \(|c-\round{c}| \leq C_0c^2 (\epsilon^{1/2} + 1/n)\). (Here, $\round{c}$ denotes the nearest integer to $c$, rounding up if $c \in \mathbb{Z}+\tfrac{1}{2}$.)
\end{theorem}

This theorem says that a Boolean function on \(S_n\) (of small expectation) whose Fourier transform is close to being concentrated on the first two irreducible representations of \(S_n\), must be close in structure to the characteristic function of a union of 1-cosets. Equivalently, a (small) subset of \(S_n\), whose characteristic function is close (in Euclidean distance) to \(U_1\), must be close in symmetric difference to a union of \(1\)-cosets.

This statement is not `stability' in the strongest sense; in fact, `genuine stability' does not occur. A `genuine' stability result would say that a subset of $S_n$ whose characteristic function is close to $U_1$ must be close in symmetric difference to a subset of $S_n$ whose characteristic function lies {\em in} $U_1$ --- i.e., close in symmetric difference to a disjoint union of 1-cosets. A union of two {\em non-disjoint} \(1\)-cosets is not close in symmetric difference to any of these, and yet its characteristic function {\em is} close to \(U_1\). Our result says that subsets close to unions of cosets (not necessarily disjoint) are the only possibility. We therefore call it a `quasi-stability' result.

If we restrict our attention to subsets \(\mathcal{A} \subset S_n\) with size close to \((n-1)!\), Theorem \ref{thm:main} says that if the characteristic function \({\bf{1}}_{\mathcal{A}}\) is close to \(U_1\), then \(\mathcal{A}\) must be close in symmetric difference to a single \(1\)-coset. This leads to our first application: a natural proof of the following conjecture of Cameron and Ku \cite{cameronkuconj}.

\begin{conjecture}
\label{conj:weakcameronku}
There exists \(\delta >0\) such that for all \(n \in \mathbb{N}\), the following holds. If \(\mathcal{A} \subset S_n\) is an intersecting family of permutations with \(|\mathcal{A}| \geq (1-\delta)(n-1)!\), then \(\mathcal{A}\) is contained within a 1-coset of \(S_n\).
\end{conjecture}

(Recall that a family $\mathcal{A} \subset S_n$ is said to be {\em intersecting} if any two permutations in $\mathcal{A}$ agree at some point.) Conjecture \ref{conj:weakcameronku} is a rather strong stability statement for intersecting families of permutations. It was first proved by the first author in \cite{cameronkuconj} using a different method (viz., by obtaining much weaker stability information, and then using the intersecting property to `bootstrap' this information). As suggested by Hatami and Ghandehari \cite{hatami}, progress on the Cameron--Ku conjecture has indeed been linked to a greater understanding of a kind of stability phenomenon for Boolean functions on \(S_n\).

As a second application, we obtain a structural description of subsets of \(S_n\) of various sizes which are almost-extremal for the edge-isoperimetric inequality for \(S_n\). If $\mathcal{A} \subset S_n$, we let $\partial \mathcal{A}$ denote the edge-boundary of $\mathcal{A}$ in the {\em transposition graph}, the Cayley graph on $S_n$ generated by the transpositions. We prove the following.

\begin{theorem}
\label{thm:roughstability}
For each \(c \in \mathbb{N}\), there exists \(n_0(c) \in \mathbb{N}\) and \(\delta_0(c)>0\) such that the following holds. Let \(\mathcal{A} \subset S_n\) with \(|\mathcal{A}| = c(n-1)!\), and with
\[|\partial \mathcal{A}| \leq \frac{|\mathcal{A}|(n!-|\mathcal{A}|)}{(n-1)!} + \delta n |\mathcal{A}|,\]
where \(n \geq n_0(c)\) and \(\delta \leq \delta_0(c)\). Then there exists a family \(\mathcal{B} \subset S_n\) such that \(\mathcal{B}\) is a union of \(c\) 1-cosets of \(S_n\), and
\[|\mathcal{A} \setminus \mathcal{B}| \leq O(c\delta)(n-1)! + O(c^2)(n-2)!.\]
(We may take \(\delta_0(c) = \Omega(c^{-4})\) and \(n_0(c) = O(c^2)\).)
\end{theorem} 

Here, the almost-extremal sets include unions of \(1\)-cosets which are not disjoint, whereas the extremal sets consist only of disjoint unions of \(1\)-cosets. We feel that this is a good example of a problem where the class of almost-extremal sets is considerably richer than the class of extremal sets.

This paper is part one of a `trilogy' dealing with results similar to Theorem \ref{thm:main}. In \cite{EFF2}, we deal with {\em balanced} Boolean functions whose Fourier transform is highly concentrated on the first two irreducible representations of $S_n$. We prove that if $f$ is a Boolean function on $S_n$ with expectation bounded away from 0 and 1, with Fourier transform highly concentrated on the first two irreducible representations of $S_n$, then $f$ is close in structure to a dictatorship. Hence, in the balanced case, genuine stability occurs, as opposed to the `quasi-stability' phenomenon for Boolean functions with expectation $O(1/n)$, in the current paper.
In both cases, however, the Boolean functions which are close to $U_1$ are close in structure to unions of 1-cosets.
The reason for the disparity between the balanced case and the sparse case is that a union of $c$ pairwise non-disjoint 1-cosets is $\Theta(c^2/n^2)$-far from $U_1$. In the setting of the current paper, $c = o(n)$ and so $\Theta(c^2/n^2) = o(1)$. By contrast, in the setting of~\cite{EFF2}, $c = \Theta(n)$ and so $\Theta(c^2/n^2) = \Theta(1)$, so a union of 1-cosets can be close to $U_1$ only if it is essentially a disjoint union of 1-cosets.
Our approaches in the two cases are completely different, and to date, we have not been able to come up with a unified approach works for the entire range $1/n \le \mathbb{E}(f) \le 1-1/n$.

The third part of our trilogy, \cite{EFF3}, deals with Boolean functions on $S_n$ whose Fourier transform is highly concentrated on irreducible representations corresponding to partitions of $n$ with first row of length at least $n-t$, or equivalently, Boolean functions which are `close' in Euclidean distance to the subspace of $\mathbb{C}[S_n]$ spanned by all characteristic functions of $t$-cosets. We prove that such a function must be `close' to the characteristic function of a union of $t$-cosets, using methods similar to the ones in this paper. However, the amount of representation theory needed for this makes for a hefty treatise which deserves a separate showcase. We point out that this is analogous to the state of affairs in the theory of Boolean functions on $\{0,1\}^n$. There, the theorems dealing with Boolean functions whose Fourier transform is highly concentrated on sets of size at most $t$, for $t >1$ (e.g.\ \cite{Bourgain}, \cite{FriedgutJunta}, \cite{KindlerSafra}, \cite{NisanSzegedy}, and recently \cite{KindlerODonnell}), tend to be far more complicated than in the $t=1$ case. In the $\{0,1\}^n$ case, such theorems have proven to be quite useful, e.g.\ as an important component in the proof of a `stability' version of the Simonovits-S\'{o}s conjecture \cite{EFF}. We trust that the symmetric-group versions will prove useful too.

The structure of the rest of the paper is as follows. In section 2 we provide some general background and notation. In section 3 we state and prove our main theorem. In section 4 we describe our two applications. Finally, in section 5 we mention some open problems.

\section{Notation and Background}
\subsection{General representation theory}
In this section, we recall the basic notions and results we need from general representation theory. For more background, the reader may consult \cite{serre}.

Let \(G\) be a finite group. A {\em representation of \(G\) over \(\mathbb{C}\)} is a pair \((\rho,V)\), where \(V\) is a finite-dimensional complex vector space, and \(\rho\colon G \to GL(V)\) is a group homomorphism from \(G\) to the group of all invertible linear endomorphisms of \(V\). The vector space \(V\), together with the linear action of \(G\) defined by \(gv = \rho(g)(v)\), is sometimes called a \(\mathbb{C}G\)-{\em module}.  A {\em homomorphism} between two representations \((\rho,V)\) and \((\rho',V')\) is a linear map \(\phi\colon V \to V'\) such that \(\phi(\rho(g)(v)) = \rho'(g)(\phi(v))\) for all \(g \in G\) and \(v \in V\). If \(\phi\) is a linear isomorphism, the two representations are said to be {\em equivalent}, or {\em isomorphic}, and we write \((\rho,V) \cong (\rho',V')\). If \(\dim(V)=n\), we say that \(\rho\) {\em has dimension} \(n\), and we write \(\dim(\rho) = n\).

The representation \((\rho,V)\) is said to be {\em irreducible} if it has no proper subrepresentation, i.e. there is no proper subspace of \(V\) which is \(\rho(g)\)-invariant for all \(g \in G\).

It turns out that for any finite group \(G\), there are only finitely many equivalence classes of irreducible complex representations of \(G\), and {\em any} complex representation of \(G\) is isomorphic to a direct sum of irreducible representations of \(G\). Hence, we may choose a set of representatives \(\mathcal{R}\) for the equivalence classes of complex irreducible representations of \(G\).

If \((\rho,V)\) is a complex representation of \(V\), the {\em character} \(\chi_{\rho}\) of \(\rho\) is the map defined by
\begin{eqnarray*}
\chi_{\rho}\colon  G & \to & \mathbb{C};\\
 \chi_{\rho}(g) &=& \textrm{Tr} (\rho(g)),
\end{eqnarray*}
where \(\textrm{Tr}(\alpha)\) denotes the trace of the linear map \(\alpha\) (i.e. the trace of any matrix of \(\alpha\)). Note that \(\chi_{\rho}(\textrm{Id}) = \dim(\rho)\), and that \(\chi_{\rho}\) is a {\em class function} on \(G\) (meaning that it is constant on each conjugacy-class of \(G\).)

The usefulness of characters lies in the following
\begin{fact}
Two complex representations are isomorphic if and only if they have the same character.
\end{fact}
\begin{definition}
Let \(\mathcal{R}\) be a complete set of non-isomorphic, irreducible representations of \(G\), i.e. containing one representative from each isomorphism class of irreducible representations of \(G\). Let \(f:G \to \mathbb{C}\) be a complex-valued function on \(G\). The {\em Fourier transform} of \(f\) is defined by
\begin{equation}
 \Hat{f}(\rho) = \frac {1}{|G|} \sum_{\sigma \in G} f(\sigma)\rho(\sigma) \quad (\rho \in \mathcal{R});
\end{equation}
it can be viewed as a map from $\mathcal{R}$ to \(\textrm{End}(V)\), the space of all linear endomorphisms of $V$.
\end{definition}

Let \(G\) be a finite group. Let \(\mathbb{C}[G]\) denote the vector space of all complex-valued functions on \(G\). Let \(\mathbb{P}\) denote the uniform probability measure on \(G\):
\[\mathbb{P}(\mathcal{A}) = |\mathcal{A}|/|G|\quad (\mathcal{A} \subset G).\]

We equip \(\mathbb{C}[G]\) with the inner product induced by the uniform probability measure on \(G\):
\[\langle f,g \rangle = \frac{1}{|G|}\sum_{\sigma \in G} f(\sigma) \overline{g(\sigma)}.\]
Let
\[||f||_2 = \sqrt{\mathbb{E}[f^2]} = \sqrt{\frac{1}{|G|}\sum_{\sigma \in G} |f(\sigma)|^2}\]
denote the induced Euclidean norm.

For each irreducible representation \(\rho\) of \(G\), let
\[U_{\rho} := \{f \in \mathbb{C}[G]:\ \hat{f}(\pi) = 0\ \textrm{for all irreducible representations } \pi \ncong \rho\}.\]

We refer to this as `the subspace of functions whose Fourier transform is supported on the irreducible representation \(\rho\)'. Note that if \(\rho' \cong \rho\), then \(U_{\rho'} = U_{\rho}\). It turns out that the \(U_{\rho}\)'s are pairwise orthogonal, and that
\[\mathbb{C}[G] = \bigoplus _{\rho \in \mathcal{R}} U_{\rho}.\]
Moreover, we have $\dim(U_\rho) = (\dim(\rho))^2$ for all $\rho$. 

For each \(\rho \in \mathcal{R}\), let \(f_{\rho}\) denote orthogonal projection onto the subspace \(U_{\rho}\). Then we have
\begin{equation}\label{eq:parseval}||f||_2^2 = \sum_{\rho \in \mathcal{R}} ||f_\rho||_2^2.\end{equation}

\subsection*{Background on the representation theory of \(S_n\).}
\label{sec:backgroundS_n}
\begin{definition}
A \emph{partition} of \(n\) is a non-increasing sequence of integers summing to \(n\), i.e. a sequence $\lambda = (\lambda_1, \ldots, \lambda_k)$ with \(\lambda_{1} \geq \lambda_{2} \geq \ldots \geq \lambda_{k}\) and \(\sum_{i=1}^{k} \lambda_{i}=n\); we write \(\lambda \vdash n\). For example, \((3,2,2) \vdash 7\).
\end{definition}
The following two orders on partitions of \(n\) will be useful.

\begin{definition}
  (Dominance order) Let $\lambda = (\lambda_1, \ldots, \lambda_r)$ and $\mu = (\mu_1,
  \ldots, \mu_s)$ be partitions of $n$. We say that $\lambda \domgeq
  \mu$ ($\lambda$ {\em dominates} $\mu$) if $\sum_{j=1}^{i} \lambda_i \geq \sum_{j=1}^{i} \mu_i \ \forall i$ (where we define \(\lambda_{i} = 0 \ \forall i > r,\ \mu_{j} = 0 \ \forall j > s\)).
\end{definition}

It is easy to see that this is a partial order. 

\begin{definition}
  (Lexicographic order) Let $\lambda = (\lambda_1, \ldots, \lambda_r)$ and $\mu = (\mu_1,
  \ldots, \mu_s)$ be partitions of $n$. We say that $\lambda > \mu$ if \(\lambda_{j} > \mu_{j}\), where \(j = \min\{i \in [n]:\ \lambda_{i} \neq \mu_{i}\}\).
\end{definition}
It is easy to see that this is a total order which extends the dominance order.

It is well-known that there is an explicit 1-1 correspondence between irreducible representations of \(S_{n}\) (up to isomorphism) and partitions of \(n\). The reader may refer to \cite{sagan} for a full description of this correspondence, or to the paper \cite{tintersecting} for a shorter description.

For each partition \(\alpha\), we write \([\alpha]\) for the corresponding isomorphism class of irreducible representations of \(S_n\), and we write \(U_{\alpha} = U_{[\alpha]}\) for the vector space of complex-valued functions on \(\Gamma\) whose Fourier transform is supported on \([\alpha]\). Similarly, if \(f \in \mathbb{C}[S_n]\), we write \(f_{\alpha}\) for the orthogonal projection of \(f\) onto \(U_{\alpha}\).

We will be particularly interested in the first two irreducible representations of \(S_n\) (under the lexicographic order on partitions). The first, \([n]\), is the trivial representation; \(U_{[n]}\) is the subspace of \(\mathbb{C}[S_n]\) consisting of the constant functions. The second may be obtained as follows.

The {\em permutation representation} \(\rho_{\textrm{perm}}\) is the representation corresponding to the permutation action of \(S_n\) on \(\{1,2,\ldots,n\}\). It turns out that \(\rho_{\perm}\) decomposes into a direct sum of a copy of the trivial representation \([n]\) and a copy of \([n-1,1]\), the second irreducible representation of \(S_n\).

As observed in \cite{tintersecting}, we have
\[U_{(n)} \oplus U_{(n-1,1)} = \textrm{Span}\{{\bf{1}}_{T_{ij}}:\ i,j \in [n]\},\]
where
\[T_{ij} = \{\sigma \in S_n: \sigma(i)=j\}.\]
The \(T_{ij}\)'s are the cosets of the point-stabilizers in \(S_n\); for brevity, we call them the {\em 1-cosets} of \(S_n\). We write
\[U_1 := U_{(n)} \oplus U_{(n-1,1)} = \textrm{Span}\{{\bf{1}}_{T_{ij}}:\ i,j \in [n]\}.\]
If \(f \in \mathbb{C}[S_n]\), we will write \(f_1\) for the orthogonal projection of \(f\) onto the subspace \(U_1\); note that \(f_1 = f_{(n)}+f_{(n-1,1)}\).

Similarly, if $t >1$, and if $I$ and $J$ are ordered $t$-tuples of distinct elements of $[n]$, then we write
$$T_{IJ} := \{\sigma \in S_n:\ \sigma(I)=J\}.$$
We call the $T_{IJ}$'s the {\em $t$-cosets} of $S_n$, and we define
$$U_t: = \textrm{Span}\{{\bf{1}}_{T_{IJ}}:\ I,J \text{ are ordered }t\text{-tuples of distinct elements of }[n]\}.$$

Recall the following theorem from \cite{tintersecting}, which completely characterizes the Boolean functions in \(U_1\).
\begin{theorem}
\label{thm:characterization}
If \(\mathcal{A} \subset S_n\) has \({\bf{1}}_{\mathcal{A}} \in U_1\), then \(\mathcal{A}\) is a disjoint union of 1-cosets of \(S_n\).
\end{theorem}

\begin{remark}
\label{rem:dict}
If $\mathcal{A} \subset S_n$ is a disjoint union of 1-cosets of $S_n$, then we either have
$$\mathcal{A} = \bigcup_{j \in J}T_{ij}$$
for some $i \in [n]$ and some $J \subset [n]$, or
$$\mathcal{A} = \bigcup_{i \in I}T_{ij}$$
for some $j \in [n]$ and some $I \subset [n]$. Hence, $\boldsymbol{1}_{\mathcal{A}}$ must be determined by the image or preimage of a single element. We may therefore call $\boldsymbol{1}_{\mathcal{A}}$ a {\em dictatorship}, by analogy with the $\{0,1\}^n$ case, hence the title of this paper.
\end{remark}

\begin{remark}
It is in place to remark that if $t \geq 2$, then a Boolean function in $U_t$ is not necessarily the characteristic function of a union of $t$-cosets. Theorem 27 in \cite{tintersecting} states that a Boolean function in $U_t$ is the characteristic function of a disjoint union of $t$-cosets, but this is false for $t \geq 2$; a counterexample, and the error in the proof, is pointed out by the second author in \cite{F-note}. A counterexample when $t=2$ is as follows. Let $n \geq 8$. For any permutation $\pi \in S_n$, define $x = x(\pi) \in \{0,1\}^4$ by $x_i = \boldsymbol{1}_{\{\pi(i) \in [4]\}}$, and consider the function
$$f:S_n \to \{0,1\}; \quad \pi \mapsto \boldsymbol{1}_{\{x_1 \geq x_2 \geq x_3 \geq x_4 \text{ or } x_1 \leq x_2 \leq x_3 \leq x_4\}}.$$
It can be checked that $f \in U_2$, but the value of $f$ clearly cannot be determined by fixing the images of at most two elements of $[n]$, so neither $f$ nor $1-f$ is a union of 2-cosets. It is easy to use $f$ to construct a counterexample for each $t \geq 3$, by considering a product of $f$ with the characteristic function of the pointwise stabilizer of a $(t-2)$-set. We note that the main application of Theorem 27 in \cite{tintersecting} was to characterize (for large $n$) the $t$-intersecting families in $S_n$ of maximum size (i.e., to characterize the cases of equality in the Deza-Frankl conjecture); fortunately, this characterization follows immediately e.g.\ from the Hilton-Milner type result of the first author in \cite{tstability}, where the proof does not depend on Theorem 27 in \cite{tintersecting} (and indeed predates the latter).
\end{remark}

Our main theorem describes what happens when \({\bf{1}}_{\mathcal{A}}\) is {\em near} \(U_1\). We will show that if \(f\colon S_n \to \{0,1\}\) is a Boolean function on \(S_n\) such that
\[\mathbb{E}[(f - f_1)^2]\]
is small, then there exists a Boolean function \(h\) such that
\[\mathbb{E}[(f-h)^2]\]
is small, and \(h\) is the characteristic function of a union of 1-cosets of \(S_n\). Note that, by (\ref{eq:parseval}), we have
\[\mathbb{E}[(f-f_1)^2] = ||f-f_1||_2^2 = \sum_{\alpha \neq (n),(n-1,1)} ||f_\alpha||^2 = \textrm{dist}(f,U_1)^2,\]
where \(\textrm{dist}\) denotes the Euclidean distance.

Our proof relies on considering the first, second and third moments of a non-negative function \(g\) which is an affine shift of the projection \(f_1\). One may compare this with the proofs of the Abelian analogues in \cite{adfs} and \cite{fkn}, where the fourth moment is considered in order to obtain structural information.

Throughout, if \(u\) and \(v\) are functions of several variables (e.g. \(n,c,\epsilon_1\)), the notation \(u = O(v)\) will mean that there exists an absolute constant \(C\) (not depending upon any of the variables) such that \(|u| \leq C|v|\) pointwise. Similarly, the notation \(u = \Omega(v)\) will mean that there exists a universal constant \(C>0\) such that \(|u| \geq C|v|\) pointwise. As usual, $\round x$ will denote $x$ rounded to the nearest integer, rounding up if $x \in \mathbb{Z}+\tfrac{1}{2}$.

\newpage

\section{The quasi-stability theorem}
In this section, we prove our main `quasi-stability' theorem, Theorem \ref{thm:main}.
\begin{theorem1}
There exist absolute constants \(C_0,\epsilon_0 >0\) such that the following holds. Let \(\mathcal{A} \subset S_n\) with \(|\mathcal{A}| = c(n-1)!\), where \(c \leq n/2\), and let \(f = {\bf{1}}_{\mathcal{A}} \colon S_n \to \{0,1\}\) be the characteristic function of \(\mathcal{A}\), so that \(\mathbb{E}[f] = c/n\). Let \(f_1\) denote orthogonal projection of \(f\) onto \(U_1 = U_{(n)} \oplus U_{(n-1,1)}\). If \(\mathbb{E}[(f-f_1)^2] \leq \epsilon c/n\), where \(\epsilon \leq \epsilon_0\), then there exists a Boolean function \(\tilde{f}:S_n \to \{0,1\}\) such that
\begin{equation}\label{eq:expbound} \mathbb{E}[(f-\tilde{f})^2] \leq C_0c^2 (\epsilon^{1/2} + 1/n)/n,\end{equation}
and \(\tilde{f}\) is the characteristic function of a union of \(\round{c}\) 1-cosets of \(S_n\). Moreover, \(|c-\round{c}| \leq C_0c^2 (\epsilon^{1/2} + 1/n)\).
\end{theorem1}
\begin{remark}
Observe that Theorem \ref{thm:main} is non-trivial only if $\epsilon = O(c^{-2})$. Unfortunately, it does not imply Theorem \ref{thm:characterization} when we take $\epsilon = 0$, due to the presence of the `extra' term $1/n$ in the right-hand side of (\ref{eq:expbound}). We conjecture in Section \ref{sec:conc} that this term can be removed (see Conjecture \ref{conj:strengthening}).
\end{remark}

While the ideas behind the proof are quite simple, the proof itself is rather long and technical. So before presenting the actual proof, we give an overview.

\begin{myproof}[Proof overview.]
The proof concentrates on analysing the matrix $B = (b_{ij})_{i,j \in [n]}$ defined by
\[b_{ij} = \frac{|\mathcal{A} \cap T_{ij}|}{(n-1)!} - \frac{|\mathcal{A}|}{n!} = n\langle f, {\bf{1}}_{T_{ij}}\rangle - \langle f,\boldsymbol{1}\rangle,\]
where \(\boldsymbol{1}\) denotes the constant function with value $1$. 

The $b_{ij}$'s turn out to be quite informative. If $\mathcal{A}$ contains $T_{ij}$, and $c = o(n)$, then $b_{ij}$ is close to $1$, whereas if $\mathcal{A} = T_{kl}$, and $(k,l) \not = (i,j)$, then $b_{ij}$ is close to $0$. This is illustrated by the following example, which is a good one to keep in mind while reading the proof overview.

If $\mathcal{A}$ is a disjoint union of $c$ 1-cosets (a dictatorship), then $B$ takes one of the following forms:
\begin{equation} \label{eq:disjoint-cosets}
\begin{aligned}
& \mspace{15mu} \begin{array}{c}
\overbrace{\hphantom{-\tfrac{n-c}{n(n-1)}\kern2\tabcolsep\cdots\kern1\tabcolsep-\tfrac{n-c}{n(n-1)}}}^{c}
\end{array} \\[-8pt]
&\begin{pmatrix}
1-\frac{c}{n} & \cdots & 1-\frac{c}{n}& -\frac{c}{n} & \cdots & -\frac{c}{n} \\
-\frac{n-c}{n(n-1)} & \cdots & -\frac{n-c}{n(n-1)} & \frac{c}{n(n-1)} & \cdots & \frac{c}{n(n-1)} \\
\vdots && \vdots & \vdots && \vdots \\
-\frac{n-c}{n(n-1)} & \cdots & -\frac{n-c}{n(n-1)} & \frac{c}{n(n-1)} & \cdots & \frac{c}{n(n-1)}
\end{pmatrix}
\end{aligned}
\end{equation}
or
\begin{equation}
\begin{array}{c}
c \left\{ \vphantom{\begin{matrix} -\frac{n-c}{n(n-1)} \\ \vdots \\ -\frac{n-c}{n(n-1)} \end{matrix}}
\right. \\ \vphantom{\begin{matrix} \frac{c}{n(n-1)} \\ \vdots \\ \frac{c}{n(n-1)} \end{matrix}}
\end{array} \hspace{-10pt}
\begin{pmatrix}
1-\frac{c}{n} & -\frac{n-c}{n(n-1)} & \cdots & -\frac{n-c}{n(n-1)} \\
\vdots & \vdots && \vdots \\
1-\frac{c}{n} & -\frac{n-c}{n(n-1)} & \cdots & -\frac{n-c}{n(n-1)} \\
-\frac{c}{n} & \frac{c}{n(n-1)} & \cdots & \frac{c}{n(n-1)} \\
\vdots & \\
-\frac{c}{n} & \frac{c}{n(n-1)} & \cdots & \frac{c}{n(n-1)}
\end{pmatrix}.
\end{equation}

Note that in both the above matrices, the sum of the squares of the entries is approximately $c$, and also the sum of the cubes of the entries is approximately $c$. Our first step will be to show that, under the hypotheses of the theorem, the same is true for $B$. This in turn will enable us to show that $B$ contains $m$ entries which are close to $1$, where $m \approx c$, and all other entries are close to $0$.

Rather than working directly with $f_1$, it turns out to be easier to work with the function
\[h = \sum_{i,j}b_{ij}{\bf{1}}_{T_{i,j}},\]
which is an affine shift of \(f_1\). This is because the second and third moments of $h$ are nicely related to the $b_{ij}$'s, whereas the same is not true of $f_1$. Indeed, it turns out (see Lemma \ref{lemma:h2h3}) that
\[\mathbb{E}[h^2] = \frac{1}{n-1} \sum_{i,j}b_{ij}^2\]
and
\[\mathbb{E}[h^3] = \frac{n}{(n-1)(n-2)} \sum_{i,j} b_{ij}^3.\]
The bound on $\mathbb{E}[(f-f_1)^2]$ gives us a bound on \(\mathbb{E}[h^2]\), and hence a bound on $\sum_{i,j}b_{ij}^2$. To obtain information about \(\mathbb{E}[h^3]\), it is helpful to consider another affine shift of \(f_1\), namely the function
\[g = \sum_{i,j} \frac{|\mathcal{A} \cap T_{ij}|}{(n-1)!} {\bf{1}}_{T_{ij}} = \frac{n}{n-1} f_1 + \frac{n-2}{n-1} c,\]
which is non-negative. We use the bound on \(\mathbb{E}[(f-f_1)^2]\), together with the fact that \(f\) is Boolean, to obtain a lower bound on \(\mathbb{E}[g^3]\), which translates to a lower bound on \(\mathbb{E}[h^3]\), finally giving us a lower bound on \(\sum_{i,j}b_{ij}^3\).

Let $F(n,c)$ and $G(n,c)$ denote respectively the sum of the squares and the sum of the cubes of the entries of the matrix (\ref{eq:disjoint-cosets}). We obtain
\beq{1} F(n,c) - c(1+\tfrac{1}{n-1})\epsilon \leq \sum_{i,j} b_{ij}^2 \leq F(n,c),
\enq
and
\beq{2} \sum_{i,j} b_{ij}^3 \geq G(n,c)-\tfrac{27}{4} (3c^2+2c) \epsilon_1^{1/2}.
\enq

Subtracting (\ref{2}) from (\ref{1}), we deduce that 
\[ \sum_{i,j} b_{ij}^2 (1 - b_{ij}) \leq O(c^2)\sqrt{\epsilon} + O(c/n). \]
This means that each \(b_{ij}\) is either very close to~$1$ or very close to~$0$. Since \(\sum_{i,j}b_{ij}^2\) is close to $F(n,c) \approx c$, it follows that there are roughly~$c$ entries which are very close to~$1$. (This implies that $c$ must be close to an integer, \(m\) say.) These entries correspond to \(m\) 1-cosets of \(S_n\), whose union is almost contained within \(\mathcal{A}\). These 1-cosets need not be disjoint, but provided \(c = o(n)\), their union has size roughly \(c(n-1)!\) (the error is of order $c^2(n-2)!$), so it gives a good approximation to \(\mathcal{A}\). This will complete the proof. 
\end{myproof}

\begin{proof}[Proof of Theorem \ref{thm:main}:]
First, note that for any absolute constant \(n_0\), we may choose \(C_0\) sufficiently large that the conclusion of the theorem holds for all \(n \leq n_0\). Hence, we may assume throughout that \(n > n_0\), for any fixed \(n_0 \in \mathbb{N}\).

Let
\[a_{ij} = \frac{|\mathcal{A} \cap T_{ij}|}{(n-1)!},\]
and let
\[b_{ij} = a_{ij}-c/n.\]
Let \(B\) denote the matrix \((b_{ij})_{i,j \in [n]}\). Note that
\[a_{ij} = n\langle f,{\bf{1}}_{T_{ij}}\rangle,\]
so
\[b_{ij} = n\langle f,{\bf{1}}_{T_{ij}}\rangle-c/n.\]
Moreover,
\begin{equation}\label{eq:cbistochastic}\sum_{j=1}^{n}b_{ij} = 0\ \textrm{for all }i \in [n],\quad \textrm{and} \quad \sum_{i=1}^{n} b_{ij}=0\ \textrm{for all }j \in [n],\end{equation}
i.e. the matrix \(B\) has all its row and column sums equal to \(0\).

Instead of working directly with the function \(f_1\), it will be convenient to work with the functions
\begin{equation}\label{eq:defg}g = \sum_{i,j}a_{ij}{\bf{1}}_{T_{ij}}\end{equation}
and
\begin{equation}\label{eq:defh}h = \sum_{i,j}b_{ij}{\bf{1}}_{T_{ij}}.\end{equation}
These are both affine shifts of \(f_1\); indeed, we have
\begin{equation}
\label{eq:lincombg}
g = n\frac{f_1+(n-2)\mathbb{E}[f_1]}{n-1} = (1+\tfrac{1}{n-1})f_1 + (1-\tfrac{1}{n-1})c,
\end{equation}
since the functions on both sides lie in \(U_1\), and they both have the same inner product with \({\bf{1}}_{T_{ij}}\), for every \(i,j \in [n]\). (Recall that \(U_1 = \textrm{Span}\{{\bf{1}}_{T_{ij}}:\ i,j \in [n]\}\).)
Observe that
\begin{equation}
\label{eq:lincombh}
h = g-c = n\frac{f_1+(n-2)\mathbb{E}[f_1]}{n-1} -c= (1+\tfrac{1}{n-1})f_1 -\tfrac{c}{n-1}.
\end{equation}

We now proceed to translate the information we know about \(f_1\) to information about \(h\); this in turn will give us information about the matrix \(B=(b_{ij})\). We have \(\mathbb{E}[f_1] = \mathbb{E}[f] = c/n\), and therefore
\[\mathbb{E}[g] = c,\]
so
\[\mathbb{E}[h]=0.\]
Write \(\mathbb{E}[(f-f_1)^2] = \epsilon_1 c/n\); by assumption, \(\epsilon_1 \leq \epsilon\). Since \(f_1\) is an orthogonal projection of \(f\), we have
\begin{equation}\label{eq:expectationf12} \mathbb{E}[f_1^2] = \mathbb{E}[f^2] - \mathbb{E}[(f-f_1)^2] = \mathbb{E}[f] - \mathbb{E}[(f-f_1)^2] = (1-\epsilon_1) c/n.\end{equation}
From (\ref{eq:lincombg}), we have
\[g^2 = (\tfrac{n}{n-1})^2(f_1^2 + 2(n-2)f_1 \mathbb{E}[f_1]+ (n-2)^2 (\mathbb{E}[f_1])^2).\]
Taking expectations, we obtain
\[\mathbb{E}[g^2] = (\tfrac{n}{n-1})^2(\mathbb{E}[f_1^2]+n(n-2)(\mathbb{E}[f_1])^2).\]
Substituting \(\mathbb{E}[f_1] = c/n\) and (\ref{eq:expectationf12}) to this expression yields:
\begin{equation}\label{eq:g2bound}\mathbb{E}[g^2] = c^2\left(1-\tfrac{1}{(n-1)^2}\right) + \left(1+\tfrac{1}{n-1}\right)^2 \tfrac{c}{n}(1-\epsilon_1).\end{equation}
Since \(h = g-\mathbb{E}[g]\), we obtain
\begin{equation}\label{eq:h2bound}\mathbb{E}[h^2] = \mathbb{E}[g^2] - (\mathbb{E}[g])^2 = \left(1+\tfrac{1}{n-1}\right)^2 \tfrac{c}{n}(1-\epsilon_1) -\tfrac{c^2}{(n-1)^2}.\end{equation}
We now proceed to express \(\mathbb{E}[h^2]\) and \(\mathbb{E}[h^3]\) in terms of the coefficients \(b_{ij}\).
\begin{lemma}
\label{lemma:h2h3}
Let \((b_{ij})_{i,j=1}^n\) be a real matrix satisfying
\begin{equation}\label{eq:cbistochastic}\sum_{j=1}^{n}b_{ij} = 0\ \textrm{for all }i \in [n],\quad and \quad \sum_{i=1}^{n} b_{ij}=0\ \textrm{for all }j \in [n],\end{equation}
and let
\[h = \sum_{i,j}b_{ij}{\bf{1}}_{T_{ij}}.\]
Then
\[\mathbb{E}[h^2] = \tfrac{1}{n-1} \sum_{i,j} b_{ij}^2,\]
and
\begin{equation}
\label{eq:b3}
\mathbb{E}[h^3] = \tfrac{n}{(n-1)(n-2)} \sum_{i,j} b_{ij}^3.
\end{equation}
\end{lemma}
\begin{proof}
We first consider \(\mathbb{E}[h^2]\). Squaring (\ref{eq:defh}), we obtain
\[h^2 = \sum_{i,j,k,l}b_{ij}b_{kl}{\bf{1}}_{T_{ij}}{\bf{1}}_{T_{kl}} = \sum_{i,j}b_{ij}^2 {\bf{1}}_{T_{ij}} +\sum_{\substack{i,j,k,l:\\ i \neq k,\ j \neq l}}b_{ij}b_{kl} {\bf{1}}_{T_{ij} \cap T_{kl}}.\]
Taking expectations, we obtain:
\begin{align}
\label{eq:h2}
\mathbb{E}[h^2] & = \tfrac{1}{n}\sum_{i,j} b_{ij}^2 + \tfrac{1}{n(n-1)} \sum_{\substack{i,j,k,l:\\ i \neq k,\ j \neq l}} b_{ij}b_{kl}\nonumber \\
&= \tfrac{1}{n}\sum_{i,j} b_{ij}^2 + \tfrac{1}{n(n-1)} \left(\sum_{i,j,k,l}b_{ij}b_{kl}-\sum_{i,j,l} b_{ij}b_{il} - \sum_{i,j,k}b_{ij}b_{kj}+\sum_{i,j}b_{ij}^2\right) \nonumber \\
&= \tfrac{1}{n-1} \sum_{i,j}b_{ij}^2,
\end{align}
where the last equality follows from (\ref{eq:cbistochastic}).

We now consider \(\mathbb{E}[h^3]\). Cubing (\ref{eq:defh}), we obtain
\begin{align*}h^3 &= \sum_{i,j,k,l,p,q}b_{ij}b_{kl}b_{pq}{\bf{1}}_{T_{ij}}{\bf{1}}_{T_{kl}}{\bf{1}}_{T_{pq}}\\
& = \sum_{i,j}b_{ij}^3 {\bf{1}}_{T_{ij}} +3\sum_{\substack{i,j,k,l:\\ i \neq k,\ j \neq l}}b_{ij}^2 b_{kl} {\bf{1}}_{T_{ij} \cap T_{kl}}\\
&+\sum_{\substack{i,k,p\ \textrm{distinct}, \\ j,l,q\ \textrm{distinct}}}b_{ij} b_{kl} b_{pq} {\bf{1}}_{T_{ij} \cap T_{kl} \cap T_{pq}}.\end{align*}
Taking the expectation of the above gives:
\begin{align}\label{eq:eh3}\mathbb{E}[g^3] &= \frac{1}{n}\sum_{i,j}b_{ij}^3 +\frac{3}{n(n-1)}\sum_{\substack{i,j,k,l:\\ i \neq k,\ j \neq l}}b_{ij}^2 b_{kl}\nonumber \\
&+\frac{1}{n(n-1)(n-2)}\sum_{\substack{i,k,p\ \textrm{distinct}, \\ j,l,q\ \textrm{distinct}}}b_{ij} b_{kl} b_{pq}.\end{align}
Observe that
\begin{align} \label{eq:sim1} \sum_{\substack{i,j,k,l:\\ i \neq k,\ j \neq l}}b_{ij}^2 b_{kl} & = \sum_{i,j,k,l} b_{ij}^2b_{kl} - \sum_{i,j,l} b_{ij}^2b_{il}-\sum_{i,j,k}b_{ij}^2b_{kj}+\sum_{i,j}b_{ij}^3 \nonumber \\
&=\sum_{i,j}b_{ij}^3,
\end{align}
using (\ref{eq:cbistochastic}).

Similarly, we have
\begin{align} \label{eq:sim2}\sum_{\substack{i,k,p\ \textrm{distinct}, \\ j,l,q\ \textrm{distinct}}}b_{ij} b_{kl} b_{pq}& = \sum_{\substack{i,k,p, \\ j,l,q}}b_{ij} b_{kl} b_{pq}(1-1_{i=k})(1-1_{k=p})(1-1_{p=i})(1-1_{j=l})(1-1_{l=q})(1-1_{q=j}) \nonumber \\
&= \sum_{\substack{i,k,p, \\ j,l,q}}b_{ij} b_{kl} b_{pq}-3\sum_{\substack{i,p\\ j,l,q}}b_{ij} b_{il} b_{pq}-3\sum_{\substack{i,k,p\\ j,q}}b_{ij} b_{kj} b_{pq} \nonumber \\
&+2\sum_{\substack{i\\ j,l,q}}b_{ij}b_{il}b_{iq}+2\sum_{\substack{i,k,p\\ j}}b_{ij}b_{kj}b_{pj}\nonumber \\
&+6\sum_{\substack{i,p\\ j,l}}b_{ij}b_{il}b_{pl}+3\sum_{\substack{i,k\\ j,l}}b_{ij}b_{kl}^2\nonumber \\
&-6\sum_{\substack{i\\ j,l}}b_{ij}b_{il}^2-6\sum_{\substack{i,p\\ j}}b_{ij}^2b_{pj}+4\sum_{i,j}b_{ij}^3\nonumber \\
&= 4\sum_{i,j}b_{ij}^3,\end{align}
again using (\ref{eq:cbistochastic}). Substituting (\ref{eq:sim1}) and (\ref{eq:sim2}) into (\ref{eq:eh3}) gives:
\begin{equation} \label{eq:h3expression} \mathbb{E}[h^3] = \frac{n}{(n-1)(n-2)} \sum_{i,j}b_{ij}^3,\end{equation}
completing the proof of Lemma \ref{lemma:h2h3}.
\end{proof}

Combining (\ref{eq:h2}) and (\ref{eq:h2bound}) yields
\begin{equation}\label{eq:sumsquares}\sum_{i,j} b_{ij}^2 = \left(1+\tfrac{1}{n-1}\right)(1-\epsilon_1)c -\tfrac{c^2}{n-1}.\end{equation}

We now require a lower bound on \(\mathbb{E}[h^3]\). In fact, it is more convenient to deal with the non-negative function \(g\); since \(g\) is an affine shift of \(h\), and \(\mathbb{E}[h]\) and \(\mathbb{E}[h^2]\) are both known, a lower bound on \(\mathbb{E}[g^3]\) will immediately yield a lower bound on \(\mathbb{E}[h^3]\).

\begin{lemma}
\label{lemma:eg3lowerbound}
Under the hypotheses of Theorem \ref{thm:main}, if \(g = (1+\tfrac{1}{n-1})f_1 + (1-\tfrac{1}{n-1})c,\) then
\[\mathbb{E}[g^3] \geq c^3 \frac{(n-2)^2(n+1)}{(n-1)^3} +3c^2 \frac{n(n-2)}{(n-1)^3}+c \frac{n^2}{(n-1)^3} - \tfrac{27}{4} (3c^2+2c) \epsilon_1^{1/2}/n.\]
\end{lemma}
\begin{proof}
Recall that 
\[\mathbb{E}[(f-f_1)^2] = \epsilon_1 c/n,\]
and
\[g = (1+\tfrac{1}{n-1})f_1 + (1-\tfrac{1}{n-1})c.\]
Let
\[F = (1+\tfrac{1}{n-1})f + (1-\tfrac{1}{n-1})c.\]
Observe that \(F\) takes just two values, \(L := (1-\tfrac{1}{n-1})c\) and \(H := (1-\tfrac{1}{n-1})c + 1+\tfrac{1}{n-1}\). Moreover, \(\mathbb{E}[F] = c\), and 
\[\mathbb{E}[(g-F)^2] = (\tfrac{n}{n-1})^2 \epsilon_1 c/n.\]
These conditions suffice to obtain a lower bound on \(\mathbb{E}[g^3]\). Indeed, we will now solve the following optimization problem.\\
\\
\textit{Problem} \(P\). Let $\theta \in (0,1)$ and let $H,L,\eta \in \mathbb{R}_{\geq 0}$ be such that $H > L$. Define a function $F\colon [0,1] \rightarrow \{H,L\}$ by
\[F(x) = \begin{cases} H & \textrm{if } 0 \leq x < \theta, \\ L & \textrm{if } \theta \leq x \leq 1. \end{cases}\]
Among all (measurable) functions $g\colon [0,1] \rightarrow \mathbb{R}_{\geq 0}$ such that $\mathbb{E}[g] = \mathbb{E}[F]$ and $\mathbb{E}[(g-F)^2] \leq \eta$, find the minimum value of $\mathbb{E}[g^3]$.\\

Observe that if \(g\colon [0,1] \to \mathbb{R}_{\geq 0}\) is feasible for \(P\), then the function
\[\tilde{g}(x) = \begin{cases} \frac{1}{\theta} \int_{0}^{\theta}g(x) \ud x & \textrm{if } 0 \leq x < \theta, \\ \frac{1}{1-\theta} \int_{\theta}^{1}g(x) \ud x & \textrm{if }\theta \leq x \leq 1, \end{cases}\]
obtained by averaging \(g\) first over \([0,\theta]\) and then over \([\theta,1]\), is also feasible. Indeed, we clearly have \(\mathbb{E}[\tilde{g}] = \mathbb{E}[g]\), and
\begin{align*}\mathbb{E}[(g-F)^2] & = \int_{0}^{\theta} (g(x)-H)^2 \ud x + \int_{\theta}^{1} (g(x)-L)^2 \ud x \\
& \geq \frac{1}{\theta}\left(\int_{0}^{\theta} (g(x) - H) \ud x \right)^2 + \frac{1}{1-\theta}\left(\int_{\theta}^{1} (g(x) - L) \ud x \right)^2\\
& = \theta \left(\frac{1}{\theta} \int_{0}^{\theta} g(x) \ud x - H \right)^2 + (1-\theta)\left(\frac{1}{1-\theta} \int_{\theta}^{1} g(x) \ud x - L \right)^2\\
& = \mathbb{E}[(\tilde{g}-F)^2],
\end{align*}
by the Cauchy-Schwarz inequality. Moreover, we have
\begin{align*}\mathbb{E}[g^3] & = \int_{0}^{\theta} g(x)^3 \ud x + \int_{\theta}^{1} g(x)^3 \ud x \\
& = \theta \cdot \frac{1}{\theta} \int_{0}^{\theta} g(x)^3 \ud x + (1-\theta) \cdot \frac{1}{1-\theta}\int_{\theta}^{1} g(x)^3 \ud x \\
& \geq \theta \left(\frac{1}{\theta}\int_{0}^{\theta} g(x) \ud x\right)^3 + (1-\theta)\left(\frac{1}{1-\theta}\int_{\theta}^{1} g(x) \ud x\right)^3\\
& = \mathbb{E}[\tilde{g}^3],
\end{align*}
by the convexity of \(y \mapsto y^3\). Hence, replacing \(g\) with \(\tilde{g}\) if necessary, we may assume that \(g\) is constant on \([0,\theta)\) and on \([\theta,1]\). In other words, we may assume that \(g\) has the following form:
\[g(x) = \begin{cases} r & \textrm{if }0 \leq x < \theta, \\ s & \textrm{if }\theta \leq x \leq 1. \end{cases}\]

Therefore, \(P\) is equivalent to the following problem:\\
\\
\textit{Problem} \(Q\):
\begin{align*} \textrm{Minimize}\quad & \theta r^3 + (1-\theta)s^3\\
\textrm{subject to}\quad & \theta r+(1-\theta)s = \theta H + (1-\theta)L,\\
& \theta (r-H)^2 + (1-\theta)(s-L)^2 \leq \eta,\\
& r,s \geq 0.\end{align*}

Or, writing \(r = H - (1-\theta)\delta\) (so that \(s = L+\theta \delta\)), we obtain the following reformulation:\\
\\
\textit{Problem} \(Q'\):
\begin{align*}\textrm{Minimize}\quad & \theta (H- (1-\theta)\delta)^3 + (1-\theta)(L+\theta \delta)^3\\
\textrm{subject to}\quad & \theta(1-\theta)\delta^2 \leq \eta,\\
& -L/\theta \leq \delta \leq H/(1-\theta).\end{align*}

When $\delta = H - L$, the function $g$ is constant. By the strict convexity of the function $y \mapsto y^3$ (on $\mathbb{R}_{\geq 0}$), the objective function is strictly decreasing on \([0,H-L]\) as a function of \(\delta\).\footnote{Alternatively, consider the derivative of the objective function, which is $3\theta(1-\theta)(\delta - (H-L)) (L + H - (1-2\theta)\delta)$.} Hence, provided \(\sqrt{\eta/(\theta(1-\theta))} \leq H-L\), the minimum is attained at \(\delta = \sqrt{\eta/(\theta(1-\theta))}\), at which point the value of the objective function is
\begin{align*}& \theta H^3 +(1-\theta)L^3  - 3(H^2-L^2) \sqrt{\theta(1-\theta)} \eta^{1/2} + 3((1-\theta) H +\theta L)\eta-\frac{1-2\theta}{\sqrt{\theta(1-\theta)}}\eta^{3/2}\\
&= \mathbb{E}[F^3] - 3(H^2-L^2) \sqrt{\theta(1-\theta)} \eta^{1/2} + 3((1-\theta) H + \theta L)\eta-\frac{1-2\theta}{\sqrt{\theta(1-\theta)}}\eta^{3/2},\end{align*}
using the fact that \(\theta H^3 +(1-\theta)L^3 = \mathbb{E}[F^3]\). Substituting in our values, namely \(\eta = (\tfrac{n}{n-1})^2 \epsilon_1 c/n\), \(H = (1-\tfrac{1}{n-1})c+(1+\tfrac{1}{n-1})\), \(L = (1-\tfrac{1}{n-1})c\), and \(\theta = c/n\), we see that provided \(\epsilon_1 \leq 1-c/n\) (which holds provided \(\epsilon_0 \leq 1/2\)), the optimum value of \(Q'\) is
\begin{align*} & \mathbb{E}[F^3] - 3(1+\tfrac{1}{n-1})\left((1+\tfrac{1}{n-1})^2 + 2c(1-\tfrac{1}{(n-1)^2})\right)(1-c/n)^{1/2} \epsilon_1^{1/2}c/n\\
& + 3\left( \left(1-\tfrac{1}{n-1}\right)c + (1-c/n)\left(1+\tfrac{1}{n-1}\right) \right) \tfrac{nc}{(n-1)^2} \epsilon_1 - \frac{1-2c/n}{\sqrt{1-c/n}}(1+\tfrac{1}{n-1})^{3} \epsilon_1^{3/2}c/n.\end{align*}
Hence, we have
\begin{align*}\mathbb{E}[g^3] & \geq \mathbb{E}[F^3] - 3(1+\tfrac{1}{n-1})\left((1+\tfrac{1}{n-1})^2 + 2c(1-\tfrac{1}{(n-1)^2})\right)(1-c/n)^{1/2} \epsilon_1^{1/2}c/n\\
& + 3\left( \left(1-\tfrac{1}{n-1}\right)c + (1-c/n)\left(1+\tfrac{1}{n-1}\right) \right) \tfrac{nc}{(n-1)^2} \epsilon_1 - \frac{1-2c/n}{\sqrt{1-c/n}}(1+\tfrac{1}{n-1})^{3} \epsilon_1^{3/2}c/n\\
& \geq \mathbb{E}[F^3] - 3(1+\tfrac{1}{n-1})^3 (2c^2+c)\epsilon_1^{1/2}/n-(1+\tfrac{1}{n-1})^3 \epsilon_1^{3/2}c/n\\
& \geq \mathbb{E}[F^3] - (1+\tfrac{1}{n-1})^3 (6c^2+4c)\epsilon_1^{1/2}/n\\
& \geq \mathbb{E}[F^3] - \tfrac{27}{4} (3c^2+2c) \epsilon_1^{1/2}/n,
\end{align*}
using the facts that \(\epsilon_1 \leq 1\), \(c/n \leq 1/2\) and \(n \geq 3\). Observe that
\[\mathbb{E}[F^3] = c^3 \frac{(n-2)^2(n+1)}{(n-1)^3} +3c^2 \frac{n(n-2)}{(n-1)^3}+c \frac{n^2}{(n-1)^3},\]
so we have
\[\mathbb{E}[g^3] \geq c^3 \frac{(n-2)^2(n+1)}{(n-1)^3} +3c^2 \frac{n(n-2)}{(n-1)^3}+c \frac{n^2}{(n-1)^3} - \tfrac{27}{4} (3c^2+2c) \epsilon_1^{1/2}/n,\]
as required.
\end{proof}

Since \(g = h+c\), we have
\[\mathbb{E}[g^3] = \mathbb{E}[(h+c)^3] = \mathbb{E}[h^3]+3c\mathbb{E}[h^2]+3c^2\mathbb{E}[h]+c^3.\]
Combining this with Lemma \ref{lemma:eg3lowerbound} yields:
\[\mathbb{E}[h^3] \geq c^3 \frac{(n-2)^2(n+1)}{(n-1)^3} +3c^2 \frac{n(n-2)}{(n-1)^3}+c \frac{n^2}{(n-1)^3} - \tfrac{27}{4} (3c^2+2c) \epsilon_1^{1/2}/n - 3c\mathbb{E}[h^2]-3c^2\mathbb{E}[h] - c^3.\]
Using \(\mathbb{E}[h]=0\) and (\ref{eq:h2bound}), we obtain:
\begin{align*}
\mathbb{E}[h^3] & \geq \frac{cn^2}{(n-1)^3}-\frac{3c^2 n}{(n-1)^3}+\frac{2c^3}{(n-1)^3}+\epsilon_1 \frac{3c^2 n}{(n-1)^2}-\tfrac{27}{4}\epsilon_1^{1/2}(3c^2+2c)/n.\\
& \geq \frac{cn^2}{(n-1)^3}-\frac{3c^2 n}{(n-1)^3}+\frac{2c^3}{(n-1)^3}-\tfrac{27}{4}\epsilon_1^{1/2}(3c^2+2c)/n.
\end{align*}
Combining this with (\ref{eq:b3}) yields:
\begin{align*}\sum_{i,j}b_{ij}^3 & \geq c\left(1-\tfrac{1}{(n-1)^2}\right) -c^2 \frac{3(n-2)}{(n-1)^2}+c^3 \frac{2(n-2)}{n(n-1)^2}-\epsilon_1^{1/2} (3c^2+2c)\frac{27(n-1)(n-2)}{4n^2}\\
& \geq c\left(1-\tfrac{1}{(n-1)^2}\right) -c^2 \frac{3(n-2)}{(n-1)^2}+c^3 \frac{2(n-2)}{n(n-1)^2}-\tfrac{27}{4}(3c^2+2c)\epsilon_1^{1/2}.
\end{align*}

To summarize, we now know that the \(b_{ij}\)'s satisfy:
\begin{equation}\label{eq:sumofsquares} \sum_{i,j} b_{ij}^2 = \left(1+\tfrac{1}{n-1}\right)(1-\epsilon_1)c -\tfrac{c^2}{n-1},\end{equation}
\begin{equation}\label{eq:sumofcubes} \sum_{i,j}b_{ij}^3 \geq c\left(1-\tfrac{1}{(n-1)^2}\right) -c^2 \frac{3(n-2)}{(n-1)^2}+c^3 \frac{2(n-2)}{n(n-1)^2}-\tfrac{27}{4}(3c^2+2c)\epsilon_1^{1/2}.\end{equation}
To illuminate the above, we now calculate the values of \(\sum_{i,j}b_{ij}^2\) and \(\sum_{i,j}b_{ij}^3\) when \(\mathcal{A}\) is a disjoint union of \(c\) 1-cosets of \(S_n\), for some \(c \in [n]\). Such a set \(\mathcal{A}\) must be of the form
\[T_{i_1 j} \cup T_{i_2 j} \cup \ldots \cup T_{i_c j}\]
for some \(j \in [n]\) and some distinct \(i_1,i_2,\ldots,i_c \in [n]\), or of the form
\[T_{ij_1}\cup T_{ij_2} \cup \ldots \cup T_{ij_c}\]
for some \(i \in [n]\) and some distinct \(i_1,i_2,\ldots,i_c \in [n]\). Clearly, all these families have the same \(\sum_{i,j}b_{ij}^2\) and the same \(\sum_{i,j} b_{ij}^3\); we may therefore assume that \(\mathcal{A} = T_{11} \cup T_{12} \cup \ldots \cup T_{1c}\). For this family, the matrix \((b_{ij})\) is as follows:
\[\left(\begin{array}{cccccc} 1-\tfrac{c}{n} & \ldots & 1-\tfrac{c}{n} & -\tfrac{c}{n} & \ldots & -\tfrac{c}{n}\\
-\tfrac{n-c}{n(n-1)} & \ldots & -\tfrac{n-c}{n(n-1)} & \tfrac{c}{n(n-1)} & \ldots & \tfrac{c}{n(n-1)}\\
\vdots && \vdots & \vdots && \vdots \\
-\tfrac{n-c}{n(n-1)} & \ldots & -\tfrac{n-c}{n(n-1)} & \tfrac{c}{n(n-1)} & \ldots & \tfrac{c}{n(n-1)}\\ \end{array} \right);\]
we have
\begin{align*} \sum_{i,j} b_{ij}^2 &= c\left(1-\tfrac{c}{n}\right)^2+c(n-1)(\tfrac{n-c}{n(n-1)})^2 + (n-c)(n-1)(\tfrac{c}{n(n-1)})^2\\
& = \frac{c(n-c)}{n-1}\\
& = c\left(1+\tfrac{1}{n-1}\right) - \tfrac{c^2}{n-1}\\
& := F(n,c),\end{align*}
and
\begin{align*}\sum_{i,j} b_{ij}^3 & = c\left(1-\tfrac{c}{n}\right)^3 - (n-c)\left(\tfrac{c}{n}\right)^3 - c(n-1)(\tfrac{c-1}{n(n-1)})^3 + (n-c)(n-1)(\tfrac{c}{n(n-1)})^3\\
& = c\left(1-\tfrac{1}{(n-1)^2}\right) -c^2 \frac{3(n-2)}{(n-1)^2}+c^3 \frac{2(n-2)}{n(n-1)^2} \\
&:= G(n,c).\end{align*}
Hence,  if \(c\) is a fixed integer, and \(\mathcal{A}\) is a family of size \(c(n-1)!\) whose characteristic function has Fourier transform which is highly concentrated on the first two levels, then (\ref{eq:sumofsquares}) says that \(\sum_{i,j}b_{ij}^2\) is close to \(F(n,c)\), the value it takes when \(\mathcal{A}\) is a disjoint union of \(c\) 1-cosets of \(S_n\). Similarly, (\ref{eq:sumofcubes}) says that \(\sum_{i,j}b_{ij}^3\) is not too far below \(G(n,c)\), the value it takes when \(\mathcal{A}\) is a disjoint union of \(c\) 1-cosets of \(S_n\). Formally, we have
\begin{equation}\label{eq:sumofsquaresF} \sum_{i,j} b_{ij}^2 = F(n,c)- c(1+\tfrac{1}{n-1})\epsilon_1,\end{equation}
and
\begin{equation}\label{eq:sumofcubesG} \sum_{i,j} b_{ij}^3 \geq G(n,c)-\tfrac{27}{4} (3c^2+2c) \epsilon_1^{1/2}.\end{equation}
These two facts will suffice to show that \(\mathcal{A}\) is close to being a union of 1-cosets of \(S_n\).

We now observe that \(c\) cannot be too small.

\begin{claim}
\label{claim:c-large}
$$c \geq 1-O(1/n)-\tfrac{135}{2} \epsilon_0^{1/2}.$$
\end{claim}
\begin{proof}[Proof of claim:]
Suppose that \(c \leq 1\). By the convexity of \(y \mapsto y^{3/2}\), we have
\begin{align*}
\sum_{i,j}b_{ij}^3 & \leq \left(\sum_{i,j}b_{ij}^2\right)^{3/2}\\
& \leq (c+2c/n)^{3/2}\\
& \leq c^{3/2}(1+2/n)^{3/2}\\
& \leq c^{3/2}(1+O(1/n)).
\end{align*}
Combining this with (\ref{eq:sumofcubes}) yields:
\begin{align*}c^{3/2}(1+O(1/n)) & \geq c-O(c/n^2) - O(c^2/n) - \tfrac{27}{4} (3c^2+2c) \epsilon_1^{1/2}\\
& \geq c(1-O(1/n) - \tfrac{135}{4}\epsilon_1^{1/2})\\
& \geq c(1-O(1/n)-\tfrac{135}{4} \epsilon_0^{1/2}).
\end{align*}
Rearranging, we obtain:
\[c^{1/2} \geq 1-O(1/n) - \tfrac{135}{4} \epsilon_0^{1/2}.\]
Squaring yields
\[c \geq 1-O(1/n) - \tfrac{135}{2} \epsilon_0^{1/2},\]
as required.
\end{proof}

Provided \(n_0\) is sufficiently large, and \(\epsilon_0\) is sufficiently small, Claim \ref{claim:c-large} implies that \(c \geq 1/2\), so \(c \leq 2c^2\). Hence, (\ref{eq:sumofsquares}) and (\ref{eq:sumofcubes}) imply the following.
\begin{equation}\label{eq:sumb2simpler}\sum_{i,j}b_{ij}^2 \leq  c+2c/n = c+O(c/n),\end{equation}
\begin{equation}\label{eq:sumb3simpler}\sum_{i,j}b_{ij}^3 \geq c - O(c^2)\epsilon_1^{1/2} -O(c^2/n).\end{equation}

Let \(x_1,\ldots,x_N\) denote the entries \((b_{ij})_{i,j \in [n]}\) in non-increasing order. We have
\begin{equation}\label{eq:sumx2simpler}\sum_{k=1}^{N}x_k^2 \leq c+O(c/n),\end{equation}
and
\begin{equation}\label{eq:sumx3simpler}\sum_{k=1}^{N}x_k^3 \geq c - O(c^2)\epsilon_1^{1/2} -O(c^2/n).\end{equation}

Subtracting \eqref{eq:sumx3simpler} from \eqref{eq:sumx2simpler} yields:
\begin{equation} \label{eq:skewness}
\sum_{k=1}^N x_k^2 (1 - x_k) \leq O(c^2) \epsilon_1^{1/2} + O(c^2/n).
\end{equation}
Let $m$ be the largest index \(k\) such that $x_k \geq 1/2$ (recall that the $x_k$ are arranged in non-increasing order). Then
\[ \sum_{k=m+1}^N x_k^2 \leq 2 \sum_{k=m+1}^N x_k^2 (1-x_k) \leq O(c^2) \epsilon_1^{1/2} + O(c^2/n).\]
Therefore
\begin{equation} \label{eq:bigterms-lb}
m \geq \sum_{k=1}^{m} x_k^2 \geq c - O(c^2) \epsilon_1^{1/2} - O(c^2/n).
\end{equation}
On the other hand, we have
\[ \sum_{k=1}^m (1 - x_k) \leq 4 \sum_{k=1}^m x_k^2 (1 - x_k) \leq O(c^2) \epsilon_1^{1/2} + O(c^2/n).\]
Rearranging,
\begin{equation} \label{eq:bigterms}
\sum_{k=1}^m x_k \geq m - O(c^2) \epsilon_1^{1/2} - O(c^2/n).
\end{equation}
Since $2x_k - 1 \leq x_k^2$, we have
\begin{equation} \label{eq:bigterms-ub}
c + O(c/n) \geq \sum_{k=1}^m x_k^2 \geq 2\sum_{k=1}^m x_k - m \geq m - O(c^2) \epsilon_1^{1/2} - O(c^2/n).
\end{equation}
Combining (\ref{eq:bigterms-lb}) and (\ref{eq:bigterms-ub}) yields:
\begin{equation} \label{eq:bigterms-int}
|c - m| \leq O(c^2) \epsilon_1^{1/2} + O(c^2/n).
\end{equation}
Our aim is now to replace $m$ by an integer $m'$ which satisfies the analogues of \eqref{eq:bigterms} and \eqref{eq:bigterms-int}, and which in addition has $|c-m'| < 1$. If $m \geq c$, then let $m' = \lceil c \rceil$. Certainly, the analogue of~\eqref{eq:bigterms-int} is satisfied, and we have
 \[\sum_{k=1}^{m'} x_k \geq \frac{m'}{m} \left(m - O(c^2) \epsilon_1^{1/2} - O(c^2/n)\right) \geq m' - O(c^2) \epsilon_1^{1/2} - O(c^2/n). 
\]
If $m < c$, then let $m' = \lfloor c \rfloor$. Again, the analogue of~\eqref{eq:bigterms-int} is satisfied, and using $x_k \geq -c/n$, we have
\[
 \sum_{k=1}^{m'} x_k \geq \sum_{k=1}^{m} x_k - c^2/n \geq m - O(c^2) \epsilon_1^{1/2} - O(c^2/n) \geq m' - O(c^2) \epsilon_1^{1/2} - O(c^2/n).
\]
Summarising, we have
\begin{gather}
 |c-m'| < \min(1, O(c^2) \epsilon_1^{1/2} + O(c^2/n)), \label{eq:bigterms-int-m} \\
 \sum_{k=1}^{m'} x_k \geq m' - O(c^2) \epsilon_1^{1/2} - O(c^2/n). \label{eq:bigterms-m}
\end{gather}
Now observe that (\ref{eq:bigterms-int-m}) and (\ref{eq:bigterms-m}) hold with $m'$ replaced by $\round c$. Indeed, we have $|c-\round c| \leq |c-m'|$ always. Moreover, if $m' \neq \round c $, then (since $|m'-c|<1$) we have 
$$1=|m'-\round c | \leq 2|m'-c| \leq O(c^2) \epsilon_1^{1/2} + O(c^2/n).$$
It follows that
$$\sum_{k=1}^{\round c} x_k \geq \round c  -1- c/n-O(c^2)\epsilon_1^{1/2} - O(c^2/n) \geq \round c -O(c^2)\epsilon_1^{1/2} - O(c^2/n),$$
since $-c/n \leq x_k \leq 1$ for all $k$. We may therefore redefine $m'=\round c $.

Let \(b_{i_lj_1},b_{i_2j_2},\ldots,b_{i_{m'} j_{m'}}\) be the entries of the matrix \(B\) corresponding to \(x_1,\ldots,x_{m'}\), and let 
\[\mathcal{C} = \bigcup_{l=1}^{m'} T_{i_l j_l}\]
denote the corresponding union of \(m'\) 1-cosets of \(S_n\). We have
\begin{align*}
\sum_{l=1}^{m'} |\mathcal{A} \cap T_{i_l j_l}| & \geq (n-1)!\sum_{l=1}^{m'} b_{i_l j_l} \\
& = (n-1)!\sum_{l=1}^{m'} x_l\\
& \geq (m'-O(c^2) \epsilon_1^{1/2} - O(c^2/n))(n-1)! \\ &\geq (c-O(c^2) \epsilon_1^{1/2} - O(c^2/n))(n-1)!.
\end{align*}
Since $|T_{i_l j_l} \cap T_{i_k j_k}| \leq (n-2)!$ whenever $k \neq l$, we have
\[
|\mathcal{A} \cap \mathcal{C}| \geq \sum_{l=1}^{m'} |\mathcal{A} \cap T_{i_l j_l}| - \binom{m'}{2} (n-2)!
\geq (c-O(c^2) \epsilon_1^{1/2} - O(c^2/n))(n-1)!,\]
i.e. \(\mathcal{A}\) contains almost all of \(\mathcal{C}\). Since \(|\mathcal{A}| = c(n-1)!\), we must have
\[|\mathcal{A} \triangle \mathcal{C}| = |\mathcal{A}| + |\mathcal{C}| - 2|\mathcal{A} \cap \mathcal{C}| \leq (O(c^2) \epsilon_1^{1/2} + O(c^2/n))(n-1)!.\]
Let \(\tilde{f} = 1_{\mathcal{C}}\) denote the characteristic function of \(\mathcal{C}\); then we have
\[\mathbb{E}[(f-\tilde{f})^2] = |\mathcal{A} \triangle \mathcal{C}|/n! \leq (O(c^2)\epsilon_1^{1/2}+O(c^2/n))/n.\]
This completes the proof of Theorem \ref{thm:main}.
\end{proof}

\section{Applications}
We now give two applications of Theorem \ref{thm:main}. The first application is a short proof of a conjecture of Cameron and Ku (Conjecture \ref{conj:weakcameronku}) on the structure of large intersecting families of permutations in $S_n$. The second application (Theorem \ref{thm:roughstability}) describes the structure of families of permutations which have small edge-boundary in the transposition graph. Both of these applications involve {\em normal Cayley graphs} on \(S_n\), so we will first give some background on normal Cayley graphs on finite groups.
\begin{definition}
Let \(G\) be a finite group, and let \(S \subset G \setminus \{\textrm{Id}\}\) be {\em symmetric} (meaning that \(S^{-1}=S\)). The {\em Cayley graph on \(G\) with generating set \(S\)} is the undirected graph with vertex-set \(G\), where we join \(g\) to \(gs\) for every \(g \in G\) and \(s \in S\); we denote it by \(\Cay(G,S)\). Formally,
\[V(\Cay(G,S)) = G,\quad E(\Cay(G,S)) = \{\{g,gs\}:\ g \in G,\ s \in S\}.\]
Note that the Cayley graph \(\Cay(G,S)\) is \(|S|\)-regular. If the generating set \(S\) is conjugation-invariant, i.e. is a union of conjugacy classes of \(G\), the Cayley graph \(\Cay(G,S)\) is said to be a {\em normal} Cayley graph.
\end{definition} 
The connection between normal Cayley graphs and the Fourier transform arises from the following fundamental theorem, which states that for any normal Cayley graph, the eigenspaces of its adjacency matrix are in 1-1 correspondence with the isomorphism classes of irreducible representations of the group.

\begin{theorem}[Frobenius / Schur / Diaconis--Shahshahani]
\label{thm:normalcayley}
Let \(G\) be a finite group, let \(S \subset G\) be an inverse-closed, conjugation-invariant subset of \(G\), let \(\Gamma\) be the Cayley graph on \(G\) with generating set \(S\), and let \(A\) be the adjacency matrix of \(\Gamma\). Let \(\{\rho_{1},\ldots,\rho_{k}\}\) be a complete set of non-isomorphic irreducible representations of \(G\) --- i.e., containing one representative from each isomorphism class of irreducible representations of \(G\). Let \(U_{\rho_i}\) denote the subspace of \(\mathbb{C}[G]\) consisting of functions whose Fourier transform is supported on \([\rho_i]\). Then we have
\[\mathbb{C}[G] = \bigoplus_{i=1}^{k}U_{\rho_i},\]
and each \(U_{\rho_i}\) is an eigenspace of \(A\) with dimension \(\dim(\rho_{i})^{2}\) and eigenvalue
\begin{equation}\label{eq:normalcayley}\lambda_{i} = \frac{1}{\dim(\rho_{i})}\sum_{g \in S} \chi_{i}(g),\end{equation}
where \(\chi_{i}(\sigma) = \textrm{Trace}(\rho_{i}(\sigma))\) denotes the character of the irreducible representation \(\rho_{i}\). 
\end{theorem}

\subsection{Large intersecting families in \(S_n\)}
In this section, we will apply Theorem \ref{thm:main} to give a new proof of a conjecture of Cameron and Ku on the structure of large intersecting families in \(S_n\).

The following definition was introduced by Deza and Frankl \cite{dezafrankl}.
\begin{definition}
We say that a family \(\mathcal{A} \subset S_n\) is {\em intersecting} if any two permutations in \(\mathcal{A}\) agree on some point --- i.e., for any \(\sigma,\pi \in \mathcal{A}\), there exists \(i \in [n]\) such that \(\sigma(i)=\pi(i)\).
\end{definition}

Deza and Frankl \cite{dezafrankl} proved the following analogue of the Erd\H{o}s--Ko--Rado theorem \cite{tekr} for permutations.

\begin{theorem}[Deza--Frankl]
\label{thm:dezafrankl}
If \(\mathcal{A} \subset S_{n}\) is intersecting, then \(|\mathcal{A}| \leq (n-1)!\).
\end{theorem}
\begin{proof} We reproduce the original proof of Deza and Frankl. Let \(\rho \in S_n\) be an \(n\)-cycle, and let \(H\) be the cyclic group of order \(n\) generated by \(\rho\). For any left coset \(\sigma H\) of \(H\), any two distinct permutations in \(\sigma H\) disagree at every point, and therefore \(\sigma H\) contains at most one member of \(\mathcal{A}\). Since the left cosets of \(H\) partition \(S_{n}\), it follows that \(|\mathcal{A}| \leq (n-1)!\).
\end{proof}
Note that equality holds in Theorem \ref{thm:dezafrankl} if $\mathcal{A}$ is a 1-coset. Deza and Frankl conjectured that equality holds only for the 1-cosets. This turned out to be much harder to prove than is usual with equality statements for Erd\H{o}s--Ko--Rado type theorems; it was eventually proved by Cameron and Ku \cite{cameron}.

\begin{theorem}[Cameron--Ku]
\label{thm:cameronku}
If \(\mathcal{A} \subset S_n\) is intersecting with \(|\mathcal{A}|=(n-1)!\), then \(\mathcal{A}\) is a 1-coset.
\end{theorem}

Larose and Malvenuto \cite{larose} independently found a different proof of Theorem \ref{thm:cameronku}. More recently, Wang and Zhang \cite{wang} gave a shorter proof. All three proofs were combinatorial; none are straightforward, all requiring a certain amount of ingenuity. In \cite{GM}, Godsil and Meagher gave an algebraic proof. In \cite{tintersecting}, a proof quite similar to that of \cite{GM} is presented.

We say that an intersecting family \(\mathcal{A} \subset S_{n}\) is \emph{centred} if there exist \(i,j \in [n]\) such that every permutation in \(\mathcal{A}\) maps \(i\) to \(j\), i.e. \(\mathcal{A}\) is contained within a 1-coset. Cameron and Ku asked how large a \textit{non-centred} intersecting family can be. Experimentation suggests that the further an intersecting family is from being centred, the smaller it must be. The following are natural candidates for large, non-centred intersecting families:
\\
\begin{itemize}
\item \(\mathcal{B}=\{\sigma \in S_{n}: \sigma \textrm{ fixes at least two elements of }\{1,2,3\}\}\).\\
\\
This has size \(3(n-2)!-2(n-3)!\).\\
It requires the removal of \((n-2)!-(n-3)!\) permutations to make it centred.\\
\item \(\mathcal{C} = \{\sigma: \sigma(1)=1,\ \sigma \textrm{ intersects } (1\ 2)\} \cup \{(1\ 2)\}\).\\
\\
\textit{Claim:} \(|\mathcal{C}| = (1-1/e+o(1))(n-1)!\).\\
\textit{Proof of Claim:} Let \(\mathcal{D}_{n} = \{\sigma \in S_{n}:\ \sigma(i)\neq i\ \forall i \in [n]\}\) be the set of \textit{derangements} of \([n]\) (permutations without fixed points); let \(d_{n} = |\mathcal{D}_{n}|\) be the number of derangements of \([n]\). By the inclusion-exclusion formula,
\[d_{n} = \sum_{i=0}^{n} (-1)^{i} {\binom{n}{i}}(n-i)! = n! \sum_{i=0}^{n}\frac{(-1)^{i}}{i!} = n!(1/e + o(1)).\]
Note that a permutation which fixes 1 intersects \((1\ 2)\) if and only if it has a fixed point greater than \(2\). The number of permutations fixing 1 alone is clearly \(d_{n-1}\); the number of permutations fixing 1 and 2 alone is clearly \(d_{n-2}\), so the number of permutations fixing 1 and some other point greater than 2 is \((n-1)!-d_{n-1}-d_{n-2}\). Hence,
\[|\mathcal{C}| = (n-1)!-d_{n-1}-d_{n-2} = (1-1/e+o(1))(n-1)!,\]
as required.\\
\\
Note that \(\mathcal{C}\) can be made centred just by removing \((1\ 2)\).
\end{itemize}
\vspace{\baselineskip}

For \(n \leq 5\), \(\mathcal{B}\) and \(\mathcal{C}\) have the same size; for \(n \geq 6\), \(\mathcal{C}\) is larger. Cameron and Ku \cite{cameron} made the following conjecture.

\begin{conjecture}
\label{conj:cameronku}
If \(n \geq 6\), and \(\mathcal{A} \subset S_n\) is a non-centred intersecting family, then \(|\mathcal{A}| \leq |\mathcal{C}|\). Equality holds only if \(\mathcal{A}\) is a `double translate' of \(\mathcal{C}\), meaning that there exist permutations \(\sigma,\pi \in S_n\) such that \(\mathcal{A} = \sigma \mathcal{C} \pi\).
\end{conjecture}

This was proved for all sufficiently large \(n\) by the first author in \cite{cameronkuconj}. To prove it, he first shows that if \(\mathcal{A} \subset S_n\) is an intersecting family with \(|\mathcal{A}| = \Omega((n-1)!)\), then the Fourier transform of \({\bf{1}}_{\mathcal{A}}\) is highly concentrated on the first two irreducible representations of \(S_n\). Secondly, he appeals to a weak version of Theorem \ref{thm:main}, namely that if \(\mathcal{A} \subset S_n\) has \(|\mathcal{A}| = \Omega((n-1)!)\), and the Fourier transform of \({\bf{1}}_{\mathcal{A}}\) is highly concentrated on the first two irreducible representations, then there exists a 1-coset \(T_{ij}\) such that \(|\mathcal{A} \cap T_{ij}| \geq \omega((n-2)!)\). Thirdly, he uses the fact that \(\mathcal{A}\) is intersecting to `bootstrap' this weak statement, showing that in fact, \(|\mathcal{A} \cap T_{ij}| \geq \Omega((n-1)!)\). This is done by showing that if $|\mathcal{A} \cap T_{ij}|$ is somewhat large, then $|\mathcal{A} \cap T_{ik}|$ must be very small for each $k \neq j$, using an extremal result on the products of the sizes of cross-intersecting families of permutations. Fourthly, he uses a combinatorial stability argument to deduce that almost all of \(T_{ij}\) is contained within \(\mathcal{A}\). Note that applying Theorem \ref{thm:main} leads to a slicker proof, as it allows us to conclude straight away that \(|\mathcal{A} \cap T_{ij}| \geq \Omega((n-1)!)\), eliminating the third stage of the argument.

Cameron and Ku also made the following weaker conjecture.
\begin{conjecture2}
There exists \(\delta >0\) such that for all \(n\), if \(\mathcal{A} \subset S_n\) is an intersecting family with \(|\mathcal{A}| \geq (1-\delta)(n-1)!\), then \(\mathcal{A}\) is contained within a 1-coset of \(S_n\).
\end{conjecture2}

Note that Conjecture \ref{conj:weakcameronku} follows immediately from Conjecture \ref{conj:cameronku}. Again, the most natural proof of Conjecture \ref{conj:weakcameronku} is via Theorem \ref{thm:main}. We give this (new) proof below.

First, we give some background on eigenvalues techniques for studying intersecting families in $S_n$.

Let \(\Gamma_n\) denote the {\em derangement graph} on \(S_n\), where two permutations are joined iff they disagree everywhere. This is simply the Cayley graph on \(S_n\) generated by the set of derangements of \([n]\). As above, we let \(\mathcal{D}_n\) denote the set of derangements of \([n]\), and we let \(d_n = |\mathcal{D}_n|\); then \(\Gamma_n\) is $d_n$-regular. Observe that an intersecting family in \(S_n\) is precisely an independent set in \(\Gamma_n\).

The following theorem of Hoffman gives an upper bound on the size of an independent set in a regular graph in terms of the eigenvalues of the adjacency matrix of the graph.

\begin{theorem}[Hoffman's theorem]
\label{thm:hoffman}
Let \(G = (V,E)\) be a \(d\)-regular graph, whose adjacency matrix \(A\) has eigenvalues \(d = \lambda_{1} \geq \lambda_{2} \geq \ldots \geq \lambda_{|V|}\).  Let \(X \subset V(G)\) be an independent set, and let \(\alpha = |X|/|V|\). Then
\[|X| \leq \frac{-\lambda_{|V|}}{d-\lambda_{|V|}}|V|.\]
If equality holds, then \({\bf{1}}_{X} - \alpha \mathbf{f} \in E_{A}(\lambda_{|V|})\), the \(\lambda_{|V|}\)-eigenspace of \(A\), where $\mathbf{f}$ denotes the all-1's vector.
\end{theorem}

Note that \(\mathcal{D}_n\) is a union of conjugacy classes of \(S_n\), so \(\Gamma_n\) is a normal Cayley graph, and therefore Theorem \ref{thm:normalcayley} can be used to calculate the eigenvalues of its adjacency matrix.

Using Theorem \ref{thm:normalcayley}, together with an analysis of symmetric functions, Renteln \cite{renteln} proved the following.

\begin{theorem}[Renteln]
\label{thm:renteln}
The minimum eigenvalue of \(\Gamma_n\) is \(-d_n/(n-1)\).
\end{theorem}

Plugging the value \(\lambda_{|V|} = -d_n/(n-1)\) into Theorem \ref{thm:hoffman} yields an alternative (much longer!) proof of Theorem \ref{thm:dezafrankl}.

In \cite{cameronkuconj}, a different proof of Theorem \ref{thm:renteln} (avoiding symmetric functions) is given; this proof shows in addition that the \(-d_n/(n-1)\) eigenspace is precisely \(U_{(n-1,1)}\). This can be used to give an alternative proof of Theorem \ref{thm:cameronku}, essentially the one presented in \cite{tintersecting} and \cite{GM}.

\begin{proof}[Proof of Theorem \ref{thm:cameronku}]
Let \(\mathcal{A} \subset S_n\) be an intersecting family of permutations with \(|\mathcal{A}|=(n-1)!\). It follows from the equality part of Hoffman's theorem that \(\mathbf{1}_{\mathcal{A}} -(|\mathcal{A}|/n!)\mathbf{f} \in U_{(n-1,1)}\). Since \(U_{(n)}\) is the space of constant functions, it follows that \(\mathbf{1}_{\mathcal{A}} \in U_{(n)}\oplus U_{(n-1,1)}=U_1\). Theorem \ref{thm:characterization} then implies that \(\mathcal{A}\) must be a disjoint union of 1-cosets of \(S_n\). Since \(\mathcal{A}\) is intersecting, it must be a single 1-coset of \(S_n\).
\end{proof}

To prove Conjecture \ref{conj:weakcameronku}, we need the following `stability version' of Hoffman's bound, proved in \cite[Lemma 3.2]{cameronkuconj}.
\begin{lemma}
\label{lemma:stabilityhoffman}
Let \(G=(V,E)\) be a \(d\)-regular graph, whose adjacency matrix \(A\) has eigenvalues \(d = \lambda_{1} \geq \lambda_{2} \geq \ldots \geq \lambda_{|V|}\). Let \(K = \max\{i: \lambda_i > \lambda_{|V|}\}\). Let \(X \subset V(G)\) be an independent set; let \(\alpha = |X|/|V|\). Let \(U = \textrm{Span}\{\mathbf{f}\} \oplus E(\lambda_{|V|})\) be the direct sum of the subspace of constant vectors and the \(\lambda_{|V|}\)-eigenspace of \(A\). Let \(P_U\) denote orthogonal projection onto the subspace \(U\). Then
\[||{\bf{1}}_X - P_U({\bf{1}}_X)||_2^2 \leq \frac{(1-\alpha)|\lambda_{|V|}| - d \alpha}{|\lambda_{|V|}|-|\lambda_{K}|}\alpha.\]
\end{lemma}

Recall the following fact from \cite{cameronkuconj}:

\begin{fact}
The derangement graph \(\Gamma_n\) has \(|\lambda_K| = O(d_n/n^2)\).
\end{fact}

Substituting this into Lemma \ref{lemma:stabilityhoffman} shows that a large intersecting family in $S_n$ must have its characteristic vector close to \(U_1\):

\begin{lemma}
\label{lemma:largeclose}
If $\mathcal{A} \subset S_n$ is an intersecting family of permutations with \(|\mathcal{A}| = \alpha n!\), then
$$
||{\bf{1}}_{\mathcal{A}} - P_{U_1}({\bf{1}}_\mathcal{A})||_2^2 \leq (1-\alpha n)(1+O(1/n))\alpha.$$
\end{lemma}
\begin{proof}
Let \(\mathcal{A} \subset S_n\) be an intersecting family with \(|\mathcal{A}| = \alpha n!\). Applying Lemma \ref{lemma:stabilityhoffman} with \(G = \Gamma_n\) and \(U = U_1 = U_{(n)} \oplus U_{(n-1,1)}\) yields:
\begin{align*}
||{\bf{1}}_{\mathcal{A}} - P_{U_1}({\bf{1}}_\mathcal{A})||_2^2 &\leq  \frac{(1-\alpha)d_{n}/(n-1) -d_{n}\alpha}{d_{n}/(n-1)-|\lambda_{K}|}\alpha\\
&=  \frac{1-\alpha -\alpha(n-1)}{1-(n-1)|\lambda_{K}|/d_{n}}\alpha\\
&=  \frac{1-\alpha n}{1-O(1/n)}\alpha\\
&=  (1-\alpha n)(1+O(1/n))\alpha,
\end{align*}
proving the lemma.
\end{proof}

We now combine Lemma \ref{lemma:largeclose} with Theorem \ref{thm:main} to show that a large intersecting family in \(S_n\) must be close to a 1-coset, an intermediate step towards proving Conjecture \ref{conj:weakcameronku}.

\begin{proposition}
\label{prop:close}
Given \(\phi>0\), there exists \(\delta = \delta(\phi) >0\) such that the following holds. If \(\mathcal{A} \subset S_n\) is an intersecting family of permutations with \(|\mathcal{A}| \geq (1-\delta)(n-1)!\), then there exists a 1-coset \(T_{ij}\) such that
\[|\mathcal{A} \triangle T_{ij}| \leq \phi(n-1)!.\]
\end{proposition}
\begin{proof}
Let $\delta>0$ to be chosen later, with $\delta<1/2$. Let \(\mathcal{A} \subset S_n\) be an intersecting family of permutations with \(|\mathcal{A}| = (1-\delta_1)(n-1)!\), where \(\delta_1 \leq \delta\). Note that, by Theorem \ref{thm:cameronku}, we may assume that \(n \geq n_0\) for some fixed \(n_0 \in \mathbb{N}\), by making \(\delta\) smaller if necessary. 

Let \(f = {\bf{1}}_{\mathcal{A}}\) denote the characteristic function of \(\mathcal{A}\), and let \(f_1 = P_{U_1}({\bf{1}}_{\mathcal{A}})\). Then, applying Lemma \ref{lemma:largeclose}, we have
\[||f-f_1||_2^2 \leq \delta_1 (1+O(1/n))(1-\delta_1)/n.\]
Hence, by Theorem \ref{thm:main}, there exists a family \(\mathcal{C}\subset S_n\) such that \(\mathcal{C} = T_{ij}\) is a 1-coset, and such that
\[|\mathcal{A} \triangle \mathcal{C}| \leq C_0 (1-\delta_1)^2 (\delta_1^{1/2}+ O(1/n))(n-1)!.\]
 Provided \(\delta\) is sufficiently small depending on \(C_0\) and \(\phi\), and \(n\) is sufficiently large depending on \(C_0\) and \(\phi\), we have
$$C_0 (1-\delta_1)^2 (\delta_1^{1/2}+ O(1/n)) \leq \phi,$$
so
\[|\mathcal{A} \triangle T_{ij}| \leq \phi (n-1)!,\]
proving the proposition.
\end{proof}

We can now give our new proof of Conjecture \ref{conj:weakcameronku}.

\begin{proof}[New proof of Conjecture \ref{conj:weakcameronku}]
Choose any \(\phi\) such that $0 < \phi < 1/e$, and let $\delta = \delta(\phi)$ be as given by Proposition \ref{prop:close}. Let \(\mathcal{A} \subset S_n\) be an intersecting family of permutations with \(|\mathcal{A}| \geq (1-\delta)(n-1)!\). Note that, by Theorem \ref{thm:cameronku}, we may assume throughout that \(n \geq n_0\), for any fixed \(n_0 \in \mathbb{N}\), by making \(\delta\) smaller if necessary.

By Proposition \ref{prop:close}, there exists a 1-coset \(T_{ij}\) such that
\begin{equation}\label{eq:oneislarge}|\mathcal{A} \triangle T_{ij}| \leq \phi(n-1)!.\end{equation}
We will show that this implies \(\mathcal{A} \subset T_{ij}\), provided $n$ is sufficiently large. Suppose for a contradiction that \(\mathcal{A} \nsubseteq T_{ij}\). Then there exists a permutation \(\tau \in \mathcal{A}\) such that \(\tau(i) \neq j\). Any permutation in \(\mathcal{A} \cap T_{ij}\) must agree with \(\tau\) at some point. But for any \(i,j \in [n]\) and any \(\tau \in S_{n}\) such that \(\tau(i) \neq j\), the number of permutations in \(S_{n}\) which map \(i\) to \(j\) and agree with \(\tau\) at some point is
\[(n-1)!-d_{n-1}-d_{n-2} = (1-1/e-o(1))(n-1)!.\]
(By double translation, we may assume that \(i=j=1\) and \(\tau = (1\ 2)\); we observed above that the number of permutations fixing \(1\) and intersecting \((1\ 2)\) is \((n-1)!-d_{n-1}-d_{n-2}\).)

Therefore, we must have
$$|\mathcal{A} \cap T_{ij}| \leq (1-1/e+o(1))(n-1)!,$$
so
$$|T_{ij} \setminus \mathcal{A}| \geq (1/e-o(1))(n-1)! > \phi(n-1)!$$
provided \(n\) is sufficiently large depending on \(\phi\), contradicting (\ref{eq:oneislarge}). This completes the proof of Conjecture \ref{conj:weakcameronku}.
\end{proof}

\subsection{Almost isoperimetric subsets of the transposition graph}
In this section, we will apply Theorem \ref{thm:main} to investigate the structure of subsets of \(S_n\) with small edge-boundary in the transposition graph. The {\em transposition graph} \(T_n\) is the Cayley graph on \(S_n\) generated by the transpositions; equivalently, two permutations are joined if, as sequences, one can be obtained from the other by transposing two elements.

In this section, we study edge-isoperimetric inequalities for \(T_n\). First, let us give some brief background on edge-isoperimetric inequalities for graphs. If \(G\) is any graph, and \(S,T \subset V(G)\), we write \(E_{G}(S,T)\) for the set of edges of \(G\) between \(S\) and \(T\), and we write \(e_{G}(S,T) = |E_{G}(S,T)|\). We write \(\partial_{G}S = E_{G}(S,S^c)\) for the set of edges of \(G\) between \(S\) and its complement; this is called the {\em edge-boundary of \(S\) in \(G\)}. An {\em edge-isoperimetric inequality for \(G\)} gives a lower bound on the minimum size of the edge-boundary of a set of size \(k\), for each integer \(k\). If \(\mathcal{A} \subset S_n\), we write \(\partial \mathcal{A} = \partial_{T_n} \mathcal{A}\) for the edge-boundary of \(\mathcal{A}\) in the transposition graph.

It would be of great interest to prove an edge-isoperimetric inequality for the transposition graph which is sharp for all set-sizes. Ben Efraim \cite{benefraim} conjectures that initial segments of the lexicographic order on \(S_n\) have the smallest edge-boundary of all sets of the same size.

\begin{definition}
If \(\sigma,\pi \in S_n\), we say that \(\sigma < \pi\) {\em in the lexicographic order on \(S_n\)} if \(\sigma(j) < \pi(j)\), where \(j = \min\{i \in [n]:\ \sigma(i) \neq \pi(i)\}\). The {\em initial segment of size \(k\) of the lexicographic order on \(S_n\)} simply means the smallest \(k\) elements of \(S_n\) in the lexicographic order.
\end{definition}

\begin{conjecture}[Ben Efraim]
\label{conj:benefraim}
For any \(\mathcal{A} \subset S_n\), \(|\partial \mathcal{A}| \geq |\partial \mathcal{C}|\), where \(\mathcal{C}\) denotes the initial segment of the lexicographic order on \(S_n\) of size \(|\mathcal{A}|\). 
\end{conjecture}

This is a beautiful conjecture; it may be compared to the edge-isoperimetric inequality in \(\{0,1\}^n\), due to Harper \cite{harper}, Lindsey \cite{lindsey}, Bernstein \cite{bernstein} and Hart \cite{hart}, stating that among all subsets of \(\{0,1\}^n\) of size \(k\), the first \(k\) elements of the binary ordering on \(\{0,1\}^n\) has the smallest edge boundary. (Recall that if \(x,y \in \{0,1\}^n\), we say that \(x < y\) {\em in the binary ordering} if \(x_j = 0\) and \(y_j = 1\), where \(j = \min\{i \in [n]:\ x_i \neq y_i\}\).)

To date, Conjecture \ref{conj:benefraim} is only known to hold for sets of size $c(n-1)!$ where $c \in \{1,2,\ldots,n\}$; this is a consequence of the work of Diaconis and Shahshahani \cite{diaconis}. (The authors have also verified it for sets of size $(n-t)!$, where $n$ is large depending on $t$; this will appear in a subsequent work, \cite{EFF3}.) Diaconis and Shahshahani proved the following isoperimetric inequality.

\begin{theorem}[Diaconis, Shahshahani]
\label{thm:diaconis}
If \(\mathcal{A} \subset S_n\), then
\begin{equation}\label{eq:diaconis}|\partial \mathcal{A}| \geq \frac{|\mathcal{A}|(n!-|\mathcal{A}|)}{(n-1)!}.\end{equation}
\end{theorem}

\begin{remark}
\label{remark:extremal}
Equality holds if and only if \(\mathcal{A}\) is a disjoint union of 1-cosets of \(S_n\) (a dictatorship).
\end{remark}

Our aim in this section is to obtain a description of subsets of $S_n$ of size $c(n-1)!$, whose edge-boundary is {\em close} to the minimum possible size (\ref{eq:diaconis}), when $c$ is small. We will prove the following.

\begin{theorem3}
For each \(c \in \mathbb{N}\), there exists \(n_0(c) \in \mathbb{N}\) and \(\delta_0(c)>0\) such that the following holds. Let \(\mathcal{A} \subset S_n\) with \(|\mathcal{A}| = c(n-1)!\), and with
\[|\partial A| \leq \frac{|\mathcal{A}|(n!-|\mathcal{A}|)}{(n-1)!} + \delta n |\mathcal{A}|,\]
where \(n \geq n_0(c)\) and \(\delta \leq \delta_0(c)\). Then there exists a family \(\mathcal{B} \subset S_n\) such that \(\mathcal{B}\) is a union of \(c\) 1-cosets of \(S_n\), and
\[|\mathcal{A} \setminus \mathcal{B}| \leq O(c\delta)(n-1)! + O(c^2)(n-2)!.\]
(We may take \(\delta_0(c) = \Omega(c^{-4})\) and \(n_0(c) = O(c^2)\).)
\end{theorem3} 

\begin{remark}
Theorem \ref{thm:roughstability} is sharp up to an absolute constant factor when \(\delta = \Omega(c/n)\); this can be seen by considering the set
\begin{align*} \mathcal{A} & = T_{1,1} \cup T_{1,2} \cup \ldots \cup T_{1,c} \cup (T_{1,c+1} \cap (T_{2,n} \cup T_{2,n-1} \cup \ldots \cup T_{2,n-k+1})) \\
& \setminus (T_{1,c} \cap (T_{2,n} \cup T_{2,n-1} \cup \ldots \cup T_{2,n-k+1})),\end{align*}
where \(n/2 \geq k = \Omega(c^2)\).
\end{remark}

Note that Theorem \ref{thm:roughstability} is not a `genuine' stability result; as with Theorem \ref{thm:main}, we may call it a `quasi-stability' result. A `genuine' stability result would say that if $\mathcal{A}$ satisfies the hypothesis of Theorem \ref{thm:roughstability}, i.e. if it has edge-boundary close to the minimum possible size, then $\mathcal{A}$ is close to an extremal family --- and, by Remark \ref{remark:extremal}, the extremal families are precisely the dictatorships, i.e. disjoint unions of 1-cosets. However, such a statement is false when $c=2$. To see this, let
\[\mathcal{A}=T_{11} \cup T_{22} \cup (T_{12} \cap T_{21});\]
then \(|\mathcal{A}|=2(n-1)!\) and
\[|\partial \mathcal{A}| = 2n(n-2)(n-2)! = (1+1/n) \frac{|\mathcal{A}|(n!-|\mathcal{A}|)}{(n-1)!} =  \frac{|\mathcal{A}|(n!-|\mathcal{A}|)}{(n-1)!} + \delta n |\mathcal{A}|,\]
where $\delta = (n-2)/n^2$, so $\mathcal{A}$ has edge-boundary very close to the minimum possible size. However, we have
\[|\mathcal{A} \Delta \mathcal{C}| \geq (n-1)! - (n-2)! = (\tfrac{1}{2}-\tfrac{1}{2(n-1)})|\mathcal{A}|\]
whenever \(\mathcal{C}\) is a dictatorship, i.e. \(\mathcal{A}\) is far (in symmetric difference) from any {\em disjoint} union of cosets. On the other hand, \(\mathcal{A}\) is close to the union of (non-disjoint) cosets \(T_{11} \cup T_{22}\); indeed,
\[|\mathcal{A} \setminus (T_{11} \cup T_{22})| = (n-2)! = \tfrac{2}{n-1}|\mathcal{A}|.\]
This is consistent with Theorem \ref{thm:roughstability}.

To prove Theorem \ref{thm:roughstability}, we will first use an eigenvalue stability argument to show that if $\mathcal{A}$ satisfies the hypothesis of the theorem (i.e. its edge-boundary has size close to the minimum possible size), then its characteristic vector $\mathbf{1}_{\mathcal{A}}$ must be close in Euclidean distance to the subspace $U_1$. We will then use Theorem \ref{thm:main} to deduce that $\mathcal{A}$ must be somewhat close (in symmetric difference) to a family $\mathcal{B} \subset S_n$ of the same size, which is a union of 1-cosets. Finally, we will use a combinatorial stability argument to deduce that $\mathcal{A}$ must be very close to $\mathcal{B}$, completing the proof of the theorem.

We proceed to give the necessary background for our initial eigenvalue stability argument. Along the way, we will show how to prove Theorem \ref{thm:diaconis}, essentially reproducing the original proof of Diaconis and Shahshahani).

Recall that if \(G = (V,E)\) is a graph, the {\em adjacency matrix} of \(G\) is the \(|V| \times |V|\) matrix \(A\) with rows and columns indexed by \(V\), and with
\[A_{u,v} = \left\{\begin{array}{cc} 1 & \textrm{if }uv \in E(G)\\
0 & \textrm{if } uv \notin E(G).\end{array}\right.\]
The {\em Laplacian matrix} \(L\) of \(G\) may be defined by
\[L = D-A,\]
where \(D\) is the diagonal \(|V| \times |V|\) matrix with rows and columns indexed by \(V\), and with 
\[D_{u,v} = \left\{\begin{array}{cc} \deg(v) & \textrm{if }u=v\\
0 & \textrm{if } u \neq v.\end{array}\right.\]

The following theorem supplies an edge-isoperimetric inequality for a graph \(G\) in terms of the eigenvalues of its Laplacian matrix.

\begin{theorem}[Dodziuk \cite{dodziuk}, Alon-Milman \cite{alonmilman}]
\label{thm:alon}
If \(G = (V,E)\) is any graph, \(L\) is the Laplacian matrix of \(G\), and \(0=\mu_1 \leq \mu_2 \leq \ldots \leq \mu_{|V|}\) are the eigenvalues of \(L\) (repeated with their multiplicities), then for any set \(S \subset V(G)\),
\[e(S,S^c) \geq \mu_2\frac{|S||S^c|}{|V|}.\]
If equality holds, then the characteristic vector \({\bf{1}}_{S}\) of \(S\) satisfies
\[{\bf{1}}_{S}- \frac{|S|}{|G|}\mathbf{f} \in \ker(L - \mu_2 I),\]
where \(\mathbf{f}\) denotes the all-1's vector.
\end{theorem}

We will show below (Lemma \ref{lemma:evalues}) how to calculate the value of $\mu_2$ for the transposition graph; plugging this value into Theorem \ref{thm:alon} will yield a proof of Theorem \ref{thm:diaconis}.

To investigate the structure of subsets with small edge-boundary in the transposition graph, we will need the following `stability version' of Theorem \ref{thm:alon}.

\begin{lemma}
\label{lemma:alonstability}
Let \(G = (V,E)\) be a graph, let \(L\) be the Laplacian matrix of \(G\), and let \(0=\mu_1 \leq \mu_2 \leq \ldots \leq \mu_{|V|}\) be the eigenvalues of \(L\) (repeated with their multiplicities). Let \(S \subset V(G)\) with
\[e(S,S^c) \leq \mu_2\frac{|S||S^c|}{|V|}+\gamma |S|,\]
where \(\gamma \geq 0\). Equip \(\mathbb{R}^V\) with the inner product
\[\langle f,g \rangle = \frac{1}{|V|}\sum_{v \in V} f(v)g(v),\]
and let
\[||f||_2 = \sqrt{\frac{1}{|V|} \sum_{v \in V} f(v)^2}\]
denote the induced Euclidean norm. Let \(M = \min\{i:\ \mu_i > \mu_2\}\). Let \(U\) denote the direct sum of the \(\mu_1\) and \(\mu_2\) eigenspaces of \(L\), and let \(P_{U}\) denote orthogonal projection onto \(U\). Then we have   
\[||{\bf{1}}_S - P_{U}({\bf{1}}_S)||_2^2 \leq \frac{\gamma}{\mu_M - \mu_{2}} \frac{|S|}{|V|}.\]
\end{lemma}
\begin{proof}
Recall that for any vector \(x \in \mathbb{R}^{V}\), we have
\[\langle x, Lx \rangle = \frac{1}{|V|}\sum_{ij \in E(G)} (x_i-x_j)^2.\]
Hence, in particular,
\[\langle {\bf{1}}_S,L{\bf{1}}_{S} \rangle = \frac{e(S,S^c)}{|V|}.\]
Let \(u_1,u_2,\ldots,u_{|V|}\) denote an orthonormal basis of eigenvectors of \(L\) corresponding to the eigenvalues \(\mu_1,\mu_2,\ldots,\mu_{|V|}\), where \(u_1 = \mathbf{f}\) is the all-1's vector. Write \({\bf{1}}_S\) as a linear combination of these basis vectors:
\[{\bf{1}}_S = \sum_{i=1}^{|V|} \xi_i u_i.\]
Let \(\alpha = |S|/|V|\) denote the measure of \(S\). Observe that
\[\xi_1 = \alpha,\quad \sum_{i=1}^{|V|} \xi_i^2 = \alpha,\quad ||{\bf{1}}_S - P_{U}({\bf{1}}_S)||_2^2 = \sum_{i \geq M} \xi_i^2.\]
Write
\[\phi = ||{\bf{1}}_S - P_{U}({\bf{1}}_S)||_2^2.\]
We have
\[\frac{e(S,S^c)}{|V|} = \langle {\bf{1}}_S,L{\bf{1}}_{S} \rangle = \sum_{i=1}^{|V|} \mu_i \xi_i^2 \geq \mu_2 (\alpha - \alpha^2 - \phi) + \mu_M \phi = \mu_2 \alpha(1-\alpha) + \phi(\mu_M-\mu_2).\]
Hence, if \(e(S,S^c) \leq \mu_2\frac{|S||S^c|}{|V|}+\gamma |S|\), then
\[\mu_2 \alpha(1-\alpha) + \phi(\mu_M-\mu_2) \leq \frac{e(S,S^c)}{|V|} \leq \mu_2 \alpha (1-\alpha) +\gamma \alpha.\]
Rearranging yields
\[\phi \leq \frac{\gamma}{\mu_M - \mu_2}\alpha,\]
as required.
\end{proof}

We now proceed to calculate \(\mu_2\) and \(\mu_M\) for the transposition graph.

\begin{lemma}
\label{lemma:evalues}
The transposition graph on $S_n$ has \(\mu_2 = n\) (for all $n \geq 2$) and $\mu_M=2n-2$ (for all $n \geq 4$). The \(0\)-eigenspace of its Laplacian is \(U_{(n)}\), and provided $n \geq 4$, the \(n\)-eigenspace is \(U_{(n-1,1)}\).
\end{lemma}

\begin{proof}
If \(G = (V,E)\) is a \(d\)-regular graph, then its Laplacian matrix is given by \(L = dI-A\). Therefore, if the eigenvalues of its adjacency matrix are
\[d = \lambda_1 \geq \lambda_2 \geq \ldots \geq \lambda_{|V|},\]
then \(\mu_i = d-\lambda_i\) for each \(i\). The transposition graph on $S_n$ is ${n \choose 2}$-regular, so it has $\mu_i = {n \choose 2}-\lambda_i$ for each \(i\). 

Note that the transposition graph is a normal Cayley graph, and therefore Theorem \ref{thm:normalcayley} applies to its adjacency matrix. Recall from section \ref{sec:backgroundS_n} that there is an explicit 1-1 correspondence between isomorphism classes of irreducible representations of \(S_n\) and partitions of \(n\); given a partition \(\alpha\), we write \(\chi_{\alpha}\) for the character of the corresponding irreducible representation of \(S_n\).

Frobenius gave the following formula for the value of \(\chi_{\alpha}\) at a transposition.
\[\chi_{\alpha}((1 \ 2)) = \frac{\dim(\rho_{\alpha})}{{\binom{n}{2}}} \tfrac{1}{2} \sum_{j=1}^{l}((\alpha_j-j)(\alpha_j-j+1) - j(j-1))\quad (\alpha = (\alpha_1,\ldots,\alpha_l) \vdash n).\]
Combining this with (\ref{eq:normalcayley}) on page \pageref{eq:normalcayley} yields the following formula for the eigenvalues of the adjacency matrix of the transposition graph:
\[\lambda_{\alpha} = \tfrac{1}{2} \sum_{j=1}^{l}((\alpha_j-j)(\alpha_j-j+1) - j(j-1))\quad (\alpha \vdash n).\]

Note that \(\lambda_{(n)} = \binom{n}{2}\) and \(\lambda_{(n-1,1)} = \binom{n}{2}-n\). Diaconis and Shashahani \cite{diaconis} verify that if \(\alpha\) and \(\alpha'\) are two partitions of \(n\) with \(\alpha \domgeq \alpha'\), then \(\lambda_{\alpha} \geq \lambda_{\alpha'}\). Since \((n-1,1) \domgeq \alpha\) for all \(\alpha \neq (n)\), we have \(\lambda_{\alpha} \leq \lambda_{(n-1,1)}\) for all \(\alpha \neq (n)\), and therefore \(\lambda_2 = \binom{n}{2}-n\). Hence, the 0-eigenspace of the Laplacian is precisely \(U_{(n)}\), the space of constant functions, and we have \(\mu_2=n\).

Note also that \(\lambda_{(n-2,2)} = \binom{n}{2}-2n+2\). Since \((n-2,2) \domgeq \alpha\) for all \(\alpha \neq (n),(n-1,1)\), we have \(\lambda_{\alpha} \leq \binom{n}{2}- 2n+2\) for all \(\alpha \neq (n),(n-1,1)\). Provided $n \geq 4$, we have
$${n \choose 2} - \lambda_{(n-1,1)} = n < 2n-2 = {n \choose 2} - \lambda_{(n-2,2)} \leq {n \choose 2}-\lambda_\alpha\quad \forall \alpha \neq (n),(n-1,1),$$
and therefore $\mu_M = 2n-2$, and the \(n\)-eigenspace of \(L\) is precisely \(U_{(n-1,1)}\).
\end{proof}

Theorem \ref{thm:diaconis} follows immediately by plugging \(\mu_2 = n\) into Theorem \ref{thm:alon}. We can also use the equality part of Theorem \ref{thm:alon} to deduce Remark \ref{remark:extremal}.

\begin{corollary}
Equality holds in Theorem \ref{thm:diaconis} only if $\mathcal{A}$ is a disjoint union of 1-cosets of $S_n$ (a dictatorship).
\end{corollary}
\begin{proof}
It is easy to see that the corollary holds for all \(n \leq 3\), so we may assume that \(n \geq 4\). If equality holds in (\ref{eq:diaconis}) for \(\mathcal{A}\), then by the equality part of Theorem \ref{thm:alon}, \({\bf{1}}_{\mathcal{A}} - (|\mathcal{A}|/n!)\mathbf{f}\) lies in the \(\mu_2\)-eigenspace of \(L\), which by Lemma \ref{lemma:evalues} is precisely \(U_{(n-1,1)}\). Therefore, \({\bf{1}}_{\mathcal{A}} \in U_{(n)} \oplus U_{(n-1,1)} = U_1\). It follows from Theorem \ref{thm:characterization} that \(\mathcal{A}\) is a disjoint union of 1-cosets of \(S_n\).
\end{proof}

We now use Lemmas \ref{lemma:alonstability} and \ref{lemma:evalues} to show that a subset of $S_n$ with small edge-boundary in the transposition graph, has characteristic vector which is close to \(U_1\).

\begin{lemma}
\label{lemma:close}
Let $n \geq 4$, and let \(\mathcal{A} \subset S_n\) with \(|\mathcal{A}| = c(n-1)!\), and with
\[|\partial A| \leq \frac{|\mathcal{A}|(n!-|\mathcal{A}|)}{(n-1)!} + \delta n |\mathcal{A}|.\]
Then
\[||{\bf{1}}_\mathcal{A} - P_{U_1} ({\bf{1}}_{\mathcal{A}})||_2^2 \leq \frac{\delta n}{n-2}\frac{c}{n}.\]
\end{lemma}
\begin{proof}
By Lemma \ref{lemma:evalues}, the transposition graph has \(\mu_2 = n\) and \(\mu_{M} = 2n-2\), the \(0\)-eigenspace of its Laplacian is \(U_{(n)}\), and the \(n\)-eigenspace is \(U_{(n-1,1)}\), so the subspace $U$ in Lemma \ref{lemma:alonstability} is $U_1$. If \(\mathcal{A}\) is as in the statement of the lemma, then we may apply Lemma \ref{lemma:alonstability} with $\gamma=\delta n$, giving
\[||{\bf{1}}_\mathcal{A} - P_{U_1} ({\bf{1}}_{\mathcal{A}})||_2^2 \leq \frac{\delta n}{n-2}\frac{c}{n},\]
as required.
\end{proof}

We now combine Lemma \ref{lemma:close} and Theorem \ref{thm:main} to give the following very rough structural description of subsets of $S_n$ with small edge-boundary in the transposition graph.

\begin{proposition}
\label{prop:roughstability}
There exists $\delta_0 > 0$ such that for each $c \in \mathbb{N}$, the following holds. Let \(\mathcal{A} \subset S_n\) with \(|\mathcal{A}| = c(n-1)!\), and with
\[|\partial A| \leq \frac{|\mathcal{A}|(n!-|\mathcal{A}|)}{(n-1)!} + \delta n |\mathcal{A}|.\]
If \(\delta \leq \delta_0\), then there exists a family \(\mathcal{B} \subset S_n\) such that \(\mathcal{B}\) is a union of \(c\) 1-cosets of \(S_n\), and
\[|\mathcal{A} \triangle \mathcal{B}| \leq C_1 c^2 (\delta^{1/2} + 1/n)(n-1)!,\]
where \(C_1>0\) is an absolute constant.
\end{proposition}
\begin{proof}
Suppose that \(\mathcal{A} \subset S_n\) with \(|\mathcal{A}| = c(n-1)!\) for some \(c \in \mathbb{N}\), and with
\[|\partial \mathcal{A}| \leq \frac{|\mathcal{A}|(n!-|\mathcal{A}|)}{(n-1)!} + \delta n |\mathcal{A}|.\]
Our aim is to show that \(\mathcal{A}\) must be close to a union of \(c\) 1-cosets of \(S_n\). Note that we may assume that \(n \geq \tfrac{1}{2}C_1 c\), otherwise we have \(C_1 c^2 (n-1)!/n \geq 2c(n-1)!\), so the conclusion of the proposition holds trivially whenever \(|\mathcal{A}|=|\mathcal{B}|=c(n-1)!\).

It follows from Lemma \ref{lemma:close} that
\[||{\bf{1}}_\mathcal{A} - P_{U_1} ({\bf{1}}_{\mathcal{A}})||_2^2 \leq \frac{\delta n}{n-2}\frac{c}{n},\]
i.e. \(\mathbf{1}_{\mathcal{A}}\) is close to \(U_{1}\). Let \(f = {\bf{1}}_{\mathcal{A}}\), and let \(f_1 = P_{U_1}({\bf{1}}_{\mathcal{A}})\); then we have
\[\mathbb{E}[(f-f_1)^2] = ||f-f_1||_2^2 \leq \frac{\delta n}{n-2}\frac{c}{n},\]
so \(\mathcal{A}\) satisfies the hypotheses of Theorem \ref{thm:main} with \(\epsilon = \delta \tfrac{n}{n-2}\). Therefore, by Theorem \ref{thm:main}, there exists a family \(\mathcal{B} \subset S_n\) which is a union of \(c\) 1-cosets of \(S_n\), and
\[|\mathcal{A} \triangle \mathcal{B}| \leq C_0 c^2((\delta n/(n-2))^{1/2} + 1/n)(n-1)! \leq \sqrt{3}C_0 c^2 (\delta^{1/2} + 1/n)(n-1)!,\]
using the fact that $n \geq 3$. This proves the proposition.
\end{proof}

We will now use a combinatorial stability argument to strengthen the bounds in the conclusion of Proposition \ref{prop:roughstability}, proving Theorem \ref{thm:roughstability}.

\begin{proof}[Proof of Theorem \ref{thm:roughstability}:]
Suppose that \(\mathcal{A} \subset S_n\) with \(|\mathcal{A}| = c(n-1)!\) for some \(c \in \mathbb{N}\), and with
\[|\partial A| \leq \frac{|\mathcal{A}|(n!-|\mathcal{A}|)}{(n-1)!} + \delta n |\mathcal{A}|.\]
Then by Proposition \ref{prop:roughstability}, there exists a family \(\mathcal{B} \subset S_n\) such that \(\mathcal{B}\) is a union of \(c\) 1-cosets of \(S_n\), and
\[|\mathcal{A} \setminus \mathcal{B}| = \psi(n-1)!,\]
where
$$\psi \leq C_1 c^2 (\delta^{1/2} + 1/n) < 1/3,$$
provided $\delta \leq O(c^{-4})$ and $n \geq \Omega(c^2)$. We proceed to obtain a better upper bound on \(\psi\) in terms of \(\delta\).

Let \(\mathcal{E} = \mathcal{A} \setminus \mathcal{B}\); then \(|\mathcal{E}| = \psi (n-1)!\). Write \(B = B_1 \cup B_2 \cup \ldots \cup B_c\), where the \(B_i\)'s are 1-cosets of \(S_n\), let \(\mathcal{M} = \mathcal{B} \setminus \mathcal{A}\), and let \(\mathcal{M}_i = B_i \setminus \mathcal{A}\) denote the set of permutations in \(B_i\) which are missing from \(\mathcal{A}\). Let \(\mathcal{N}_i = \mathcal{M}_i \setminus (\cup_{j \neq i} B_j)\), and write \(|\mathcal{N}_i| = \nu_i (n-1)!\).

We now give a lower bound on \(|\partial \mathcal{A}|\) in terms of \(\psi\). Observe that
\begin{align*}|\partial \mathcal{A}| &= |\partial \mathcal{B}| + |\partial \mathcal{E}|- 2e(\mathcal{E},\mathcal{B}) - e(\mathcal{M},S_n \setminus (\mathcal{B} \cup \mathcal{E})) + e(\mathcal{M}, \mathcal{B} \setminus \mathcal{M}) + e(\mathcal{E},\mathcal{M})\\
& = |\partial \mathcal{B}| + |\partial \mathcal{E}| - 2e(\mathcal{E},\mathcal{B})- e(\mathcal{M},S_n \setminus \mathcal{B}) + e(\mathcal{M}, \mathcal{B} \setminus \mathcal{M}) + 2e(\mathcal{E},\mathcal{M})\\
& \geq |\partial \mathcal{B}| + |\partial \mathcal{E}| - 2e(\mathcal{E},\mathcal{B})- e(\mathcal{M},S_n \setminus \mathcal{B}) + e(\mathcal{M}, \mathcal{B} \setminus \mathcal{M}).
\end{align*}

By definition, we have \(\mathcal{E} \cap B_i = \emptyset\) for each \(i \in [c]\). If \(B_i = T_{pq}\) and \(\sigma \in \mathcal{E}\) then \(\sigma (p) \neq q\), and so the only neighbour of \(\sigma\) in \(B_i\) is \(\sigma (p \ \sigma^{-1} (q))\). It follows that \(e(\mathcal{E},B_i) \leq |\mathcal{E}|\) for each \(i\). Summing over all \(i\), we obtain:
\[e(\mathcal{E},\mathcal{B}) \leq \sum_{i=1}^{c} e(\mathcal{E},B_i) \leq c|\mathcal{E}| = c\psi(n-1)!.\]

Similarly, each \(\sigma \in B_i\) has at most \(n-1\) neighbours in \(S_n \setminus B_i\). Indeed, if \(B_i = T_{pq}\), then the neighbours of \(\sigma\) in \(S_n \setminus B_i\) are \(\{\sigma(p\ r):\ r \neq p\}\). It follows that
\[e(\mathcal{M},S_n \setminus \mathcal{B}) \leq (n-1)|\mathcal{M}|.\]

By Theorem \ref{thm:diaconis}, we have
\[|\partial \mathcal{E}| \geq \psi (n-1)!(n-\psi).\]

Since $\mathcal{B}$ is a union of $c$ 1-cosets of $S_n$, it is easy to see that
\[ |\mathcal{B}| \geq c(n-1)! - \binom{c}{2} (n-2)!, \]
and so Theorem~\ref{thm:diaconis} implies 
\[ |\partial \mathcal{B}| \geq c(1-\tfrac{c-1}{2(n-1)})(n-1)! (n-c) \geq c(n-1)!(n-c) - O(c^2) (n-1)!,\]
using $c < n/2$.

Finally, it remains to bound \(e(\mathcal{M}, \mathcal{B} \setminus \mathcal{M})\) from below. To do this, note first that $B_i \setminus \mathcal{M}_i = (B_i \setminus \mathcal{N}_i) \setminus \cup_{j \neq i}B_j$ for each \(i\), and therefore
\[\{E(\mathcal{N}_i,(B_i \setminus \mathcal{N}_i) \setminus \cup_{j \neq i}B_j)):\ i \in [c]\}\]
are pairwise disjoint subsets of \(E(\mathcal{M},\mathcal{B} \setminus \mathcal{M})\). Observe that for each \(i\), we have
\[e(\mathcal{N}_i,B_i \cap B_j) \leq |\mathcal{N}_i|\quad \forall j \neq i,\]
since \(\mathcal{N}_i \cap B_j = \emptyset\) for each \(j \neq i\). Hence, we have
\[e(\mathcal{N}_i, B_i \cap \cup_{j\neq i} B_j) \leq (c-1)|\mathcal{N}_i|.\]
It follows that
\[e(\mathcal{N}_i,(B_i \setminus \mathcal{N}_i) \setminus \cup_{j \neq i}B_j) \geq e(\mathcal{N}_i,B_i \setminus \mathcal{N}_i) - (c-1)|\mathcal{N}_i|.\]
Note that \(T_n[B_i]\) is isomorphic to \(T_{n-1}\), and therefore we may apply Theorem \ref{thm:diaconis} in \(S_{n-1}\) to give:
\[e(\mathcal{N}_i,B_i \setminus \mathcal{N}_i) \geq \nu_i (n-1)! (1-\nu_i)(n-1).\]
We obtain
\[e(\mathcal{M}, \mathcal{B}\setminus \mathcal{M}) \geq \sum_{i=1}^{c} \nu_i (1-\nu_i) (n-1)(n-1)! - (c-1)(n-1)!\sum_{i=1}^{c} \nu_i.\]

Since \(|\mathcal{A}| = c(n-1)!\) and \(c(n-1)! - \binom{c}{2}(n-2)! \leq |\mathcal{B}| \leq c(n-1)!\), we have \(|\mathcal{E}| - \binom{c}{2}(n-2)! \leq |\mathcal{M}| \leq |\mathcal{E}|\). Since the \(\mathcal{N}_i\)'s are pairwise disjoint subsets of \(\mathcal{M}\), we have
\begin{equation}\label{eq:upperboundsum}\sum_{i=1}^{c}|\mathcal{N}_i| \leq |\mathcal{M}| \leq |\mathcal{E}|.\end{equation}
Note also that
\[\mathcal{M} \setminus \left(\bigcup_{i=1}^{c} \mathcal{N}_i\right) \subset \bigcup_{i\neq j} (B_i \cap B_j),\]
and therefore
\[\left|\mathcal{M} \setminus \left(\bigcup_{i=1}^{c} \mathcal{N}_i\right)\right| \leq \binom{c}{2} (n-2)!.\]
Hence, we have
\[\left|\bigcup_{i=1}^{c} \mathcal{N}_i\right| \geq |\mathcal{E}| - 2\binom{c}{2}(n-2)!,\]
and therefore
\begin{equation}\label{eq:lowerboundsum}\sum_{i=1}^{c} |\mathcal{N}_i| \geq |\mathcal{E}| - 2\binom{c}{2}(n-2)!.\end{equation}
Combining (\ref{eq:upperboundsum}) and (\ref{eq:lowerboundsum}) yields
\begin{equation}\label{eq:2sidedbound} \psi - 2\binom{c}{2}/(n-1) \leq \sum_{i=1}^{c} \nu_i \leq \psi.\end{equation}
Putting everything together, we obtain
\begin{align*} |\partial \mathcal{A}| & \geq c(n-1)!(n-c) - O(c^2)(n-1)! + \psi (n-1)!(n-\psi)- 2c\psi(n-1)!\\
& - (n-1)|\mathcal{M}| + \sum_{i=1}^{c} \nu_i (1-\nu_i) (n-1)(n-1)! - (c-1)(n-1)!\sum_{i=1}^{c} \nu_i\\
& \geq c(n-1)!(n-c) - O(c^2)(n-1)! + \psi (n-1)!(n-\psi)- 2c\psi(n-1)! \\
& - (n-1)\psi(n-1)! + \sum_{i=1}^{c} \nu_i (1-\nu_i) (n-1)(n-1)! - (c-1)(n-1)!\psi\\
& \geq c(n-1)!(n-c) - O(c^2)(n-1)! - \psi(n-1)!(3c+\psi-2) \\
& + (1-1/n) n!\sum_{i=1}^{c} \nu_i(1-\nu_i)\\
& \geq c(n-1)!(n-c) + (1-1/n) n!\psi(1-\psi) - O(c^2)(n-1)!\\
& - \psi(n-1)!(3c+\psi-2), \\
& \geq c(n-1)!(n-c-1) + (1-2/n) n!\psi(1-\psi) - O(c^2)(n-1)!,\\
\end{align*}
using \(\sum_{i=1}^{c} \nu_i \leq \psi < 1/3 \), and the fact that \(y \mapsto y(1-y)\) is concave for \(y \in [0,1]\).

Hence, we have
\begin{align*}c(n-1)!(n-c-1) + (1-2/n) n!\psi(1-\psi) - O(c^2)(n-1)!\\
\leq |\partial A| \leq c(n-1)!(n-c)+\delta n|\mathcal{A}| = c(n-1)!(n-c)+cn!\delta.\end{align*}
It follows that
\[ \psi(1-\psi) \leq \frac{c\delta + c/n + O(c^2/n)}{1-2/n} \leq 3c\delta + O(c^2/n),\]
provided $n \geq 3$. Solving for \(\psi\), we obtain
\begin{equation}\label{eq:option2}\psi \geq \tfrac{1}{2}(1+\sqrt{1-12c\delta}) - O(c^2/n),\end{equation}
or
\begin{equation}\label{eq:option1} \psi \leq \tfrac{1}{2}(1-\sqrt{1-12c\delta}) + O(c^2/n) \leq 6c\delta + O(c^2/n),\end{equation}
using the inequality $1 - \sqrt{1-x} \leq x$ for $x \in [0,1]$. Provided \(n = \Omega(c^2)\), \eqref{eq:option2} cannot hold (since \(\psi < 1/3\)), and therefore (\ref{eq:option1}) must hold. Hence,
$$|\mathcal{A} \Delta \mathcal{B}| = 2|\mathcal{A} \setminus \mathcal{B}| =2\psi (n-1)! \leq (12c \delta + O(c^2/n))(n-1)!,$$
proving the theorem.
\end{proof}

\section{Conclusion and open problems}
\label{sec:conc}
Note that the conclusion of Theorem \ref{thm:main} is non-trivial only when \(n\) is sufficiently large, and \(\delta\) sufficiently small, depending on \(c\). We believe these restrictions to be artefacts of our method of proof, and we conjecture the following strengthening of Theorem \ref{thm:main}.
\begin{conjecture}\label{conj:strengthening}
There exists an absolute constant \(C_0>0\) such that the following holds. Let \(\mathcal{A} \subset S_n\) with \(|\mathcal{A}| = c(n-1)!\), where \(0 \leq c \leq n\), and let \(f = {\bf{1}}_{\mathcal{A}} : S_n \to \{0,1\}\) be the characteristic function of \(\mathcal{A}\), so that \(\mathbb{E}[f] = c/n\). Let \(f_1\) denote orthogonal projection of \(f\) onto \(U_1 = U_{(n)} \oplus U_{(n-1,1)}\). If \(\mathbb{E}[(f-f_1)^2] \leq \epsilon c/n\), then there exists a Boolean function \(h\) such that
\[\mathbb{E}[(f-h)^2] \leq C_0 \epsilon c/n,\]
and \(h\) is the characteristic function of a union of \(\round{c}\) 1-cosets of \(S_n\). Furthermore, $|c-\round{c}| \leq C_0 \epsilon$.
\end{conjecture}
\noindent Note that this would be informative for all \(c \leq n\).

Likewise, we conjecture the following strengthening of Theorem \ref{thm:roughstability}.
\begin{conjecture}
There exists an absolute constant \(C_1>0\) such that the following holds. Let \(\mathcal{A} \subset S_n\) with \(|\mathcal{A}| = c(n-1)!\) for some $c \in \mathbb{N}$, and with
\[|\partial A| \leq \frac{|\mathcal{A}|(n!-|\mathcal{A}|)}{(n-1)!} + \delta n |\mathcal{A}|.\]
Then there exists a family \(\mathcal{B} \subset S_n\) such that \(\mathcal{B}\) is a union of \(c\) 1-cosets of \(S_n\), and
\[|\mathcal{A} \Delta \mathcal{B}| \leq C_1 c\delta (n-1)!.\]
\end{conjecture} 

Ben Efraim's conjecture (Conjecture \ref{conj:benefraim}) remains one of the most natural open problems in the area. If this could be proved, it is likely that analogues of Theorem \ref{thm:roughstability} could be obtained for other set-sizes.

\subsubsection*{Acknowledgment} We wish to thank Gil Kalai for many useful conversations. We also wish to thank two anonymous referees for their careful reading of the paper and their helpful suggestions.

 \end{document}